\documentclass[11pt,twoside]{amsart}

\usepackage{amssymb}
\usepackage{verbatim}
\usepackage{hyperref}
\usepackage{graphicx, color}

\numberwithin{equation}{section}
\newtheorem{theorem}{Theorem}[section]
\newtheorem{def-prop}[theorem]{Definition-Proposition}

\newtheorem{proposition}[theorem]{Proposition}

\newtheorem{conjecture}[theorem]{Conjecture}

\newtheorem{problem}[theorem]{Problem}

\newcommand{\figone}{1}
\newcommand{\figtwo}{2}
\newcommand{\figthree}{3}

\newcounter{xrcs}[section]

\theoremstyle{definition}
\newtheorem{definition}[theorem]{Definition}
\newtheorem{exercise}[xrcs]{Exercise}
\newtheorem{remark}[theorem]{Remark}
\newtheorem{example}[theorem]{Example}

\newtheorem{question}[theorem]{Question}

\newcommand{\field}{K}
\newcommand{\Aff}[1]{\mathbb{A}^{#1}}
\newcommand{\Spec}{\operatorname{Spec}}
\newcommand{\pr}[1]{\mathbb{P}^{#1}}
\newcommand{\Proj}{\operatorname{Proj}}

\newcommand{\OO}{{\mathcal O}}
\newcommand{\mult}{\operatorname{mult}}

\begin{document}

\title[Regina Lectures]{Regina Lectures on Fat Points}

\author{Susan Cooper}
\address{Department of Mathematics\\
Central Michigan University\\
Mount Pleasant, Mich. 48859 USA}
\email{s.cooper@cmich.edu}

\author{Brian Harbourne}
\address{Department of Mathematics \\
University of Nebraska--Lincoln\\
Lincoln, NE 68588-0130, USA }
\email{bharbour@math.unl.edu}

\keywords{Fat points, polynomial rings, symbolic powers, primary decomposition, B\'ezout's Theorem, 
Hilbert functions, O-sequence, intersection multiplicity, Riemann-Roch,
divisors, blowings-up, Cremona group, quadratic transformation}
\subjclass[2000]{13-01, 13A02 , 13A15 , 13A17, 13D40 , 13F20 , 14-01, 14C20, 14N05, 14R05}

\date{October 8, 2013}

\begin{abstract}
These notes are a record of lectures given in the
Workshop on Connections Between Algebra and Geometry
at the University of Regina,
May 29--June 1, 2012. The lectures
were meant as an introduction to current research
problems related to fat points for an audience that was not expected
to have much background in commutative algebra or algebraic geometry
(although sections 8 and 9 of these notes demand somewhat more background
than earlier sections).
\end{abstract}

\thanks{
Acknowledgments: We thank the University of Regina for its thoughtful hospitality,
the (other) organizers for their work to make the Workshop a reality,
the funders for their financial support, and the participants for
their enthusiastic engagement at the Workshop.
We also are very grateful for the referee's detailed comments
and careful reading of these notes.}

\maketitle

\tableofcontents

\begin{quotation}
\begin{center}
\emph{We dedicate these notes to Tony Geramita, a wonderful mentor, colleague and friend:
the contagious joy he takes in life and in mathematics has been an inspiration for both of us.}
\end{center}
\end{quotation}

\section{Motivation}

Fat points are relevant to many areas of research.
For example, one reason
fat points are of interest in algebraic geometry is because of their connection to
linear systems: one can identify the homogeneous components 
of ideals of fat points in $\pr n$ with the spaces of 
global sections of line bundles on blowings-up of
$\pr n$ at given finite sets of points. Fat points also arise indirectly in other topics of study
in algebraic geometry, such as the study of secant varieties \cite{refCGG}.
In commutative algebra ideals of fat points give a useful class of test cases
and suggest interesting questions that can be true more generally
(see, for example, \cite{refMN}, where the authors give a conjecture
for all nonreduced zero-dimensional schemes, and as evidence 
prove it for fat points).
Fat points also arise in more applied situations, such as combinatorics and 
in interpolation problems \cite{refHS, refI}.
Regarding the latter, consider the following question. 

\begin{question}
What can we say about a function $f\in{\bf C}[x_1,\ldots,x_n]$ if 
all we know are values of $f$ and certain of
its partial derivatives at some finite set of
points $p_1,\ldots,p_s\in{\bf C}^n$? In particular:

\begin{enumerate}
\item What is the least degree among all $f$ satisfying the given data?
\item How many such $f$ are there up to some given degree $t$?
\item What is the smallest degree $t$ guaranteed to have
such an $f$, regardless of the choice of the points $p_i$? (For example,
there is a linear $f$ vanishing at three colinear points of the plane,
but not at three noncolinear points, so the least degree $t$ guaranteeing vanishing
at three points in the plane without knowing the disposition of the points is $t=2$.)
\end{enumerate}
\end{question}

These are open problems when $n\geq 2$ even in the simplest case, where we specify
points $p_1,\ldots,p_s$, and an order of vanishing $m_i$ at each point $p_i$, and ask
to find all $f\in{\bf C}[x_1,\ldots,x_n]$ such that $\operatorname{ord}_{p_i}(f)\geq m_i$
for all $i$, where, given a point $p$, 
$\operatorname{ord}_p(f)>0$ just means $f(p)=0$, 
$\operatorname{ord}_p(f)>1$ means $f(p)=0$ and $\frac{\partial f}{\partial x_i}(p)=0$ for all $i$, 
$\operatorname{ord}_p(f)>2$ means $f(p)=0$, $\frac{\partial f}{\partial x_i}(p)=0$ for all $i$, and 
$\frac{\partial^2 f}{\partial x_i\partial x_j}(p)(p)=0$ for all $i$ and $j$, and 
$\operatorname{ord}_p(f)>m$ just means $f(p)=0$ and 
$\frac{\partial^k f}{\partial x_{i_1}\cdots \partial x_{i_k}}(p)=0$
for all $i_j$ with $k\leq m$.

Alternatively, one can think of $\operatorname{ord}_p(f)$ as the least degree of 
a term of $f$ when expressed in coordinates centered at $p$. So for example,
if $p=(a_1,\ldots,a_n)$, then let $X_i=x_i-a_i$, and substitute $x_i=X_i+a_i$ into
$f$ to get $g=f(X_1+a_i,\ldots,X_n+a_n)$. Then $\operatorname{ord}_p(f)$ is the least degree of
a nonzero term of $g$, regarded as a polynomial in the $X_i$. This removes
having to deal with partial derivatives, which can be problematic 
when working over arbitrary algebraically closed fields, i.e., 
when considering $f\in \field[x_1,\ldots,x_n]$.

To further algebracize the interpolation problem, we note that $\operatorname{ord}_p(f)\geq m$
if and only if $f\in I(p)^m$, where $I(p)$ is the ideal of all polynomials
that vanish at $p$. Thus $\operatorname{ord}_{p_i}(f)\geq m_i$ for 
points $p_1,\ldots,p_s$ and orders of vanishing $m_i$ if and only if
$f$ is in the ideal $\cap I(p_i)^{m_i}$. It is convenient to use 0-cycle notation
to specify the given data, so we write $Z=m_1p_1+\cdots+m_sp_s$, which we refer to
as a fat point \emph{scheme}, and we denote $\cap I(p_i)^{m_i}$ by $I(Z)$.
(Readers familiar with schemes in the algebraic geometric sense can just
regard $Z$ as the subscheme of $\Aff{n}_{\field}$ defined by 
$I(Z)\subseteq \field[x_1,\ldots,x_n]$.)

Given $p_1\ldots,p_s\in\field^n$ and nonnegative integers
$m_1,\ldots,m_s$, we have the ideal $I=I(m_1p_1+\cdots+m_sp_s)\subseteq \field[x_1,\ldots,x_n]=A$.
Let $A_{\leq t}$ be the $\field$-vector space span of all $f\in A$ with $\deg(f)\leq t$,
and let $I_{\leq t}=I\cap A_{\leq t}$. Then we refer to the function 
$H^{\leq}_I(t)=\dim_{\field}(I_{\leq t})$
as the Hilbert function of $I$. 
Also, given any ideal $0\neq I\subseteq A$, define $\alpha(I)$ to be 
the degree of the nonzero element of $I$ of least degree.
(If $0\neq J\subseteq R=\field[x_0,\ldots,x_n]$ is a homogeneous ideal, then $\alpha(J)$
is in fact the degree of a nonzero homogeneous element of $J$ of least degree.)
We can now raise the following open problems:

\begin{problem} Consider the following problems.
\begin{enumerate}
\item Find $\alpha(I)$.
\item Find the Hilbert function $H^{\leq}_I$ of $I$.
\item Find the maximum value of $\alpha(I(m_1q_1+\cdots+m_sq_s))$
as the $q_i$ range over all choices of $s$ distinct points of $\Aff{n}$.
(The maximum occurs on a Zariski open subset of $(\Aff{n})^s$,
so we can ask: What is the maximum value of 
$\alpha(I(m_1q_1+\cdots+m_sq_s))$ for general points
$q_1,\ldots,q_s$?)
\end{enumerate}
\end{problem}

\begin{example}
All of the problems above are easy if $n=1$, using the fact that
$\field[x_1]$ is a principal ideal domain. This is the case of Lagrange interpolation.
In this case, a point $p_i$ is just an element of $\field$, so,
for example, $I=I(m_1p_1+\cdots+m_sp_s)=(x_1-p_1)^{m_1}\cdots(x_1-p_s)^{m_s}$.
The positions of the points $p_i$ do not matter: we always
have $\alpha(I)=m_1+\cdots+m_s$ and
$H^{\leq}_I(t)=\min(0,t+1-\sum_im_i)$.
\end{example}

We end this introduction with an advisory to the reader. It is common to refer to the
data $m_1p_1+\cdots+m_sp_s$ as being points $p_i$ with \emph{multiplicities} $m_i$.
This grows out of the universal terminology that a root of a polynomial in a single variable
can be a multiple root; for example, $x=1$ is a root of multiplicity 2 for $f(x)=x^2-2x+1$.
This terminology is quite old (see \cite{refSyl,refN2}, for example).
More recently, commutative algebraists have used multiplicity to refer to
what can also be called the degree of
a fat point subscheme. In this sense, the multiplicity of
$mp$ for a point $p\in\Aff{n}$ is $\binom{m+n-1}{n}$
(see \cite[p.\ 66]{refEH}, for example).
Regardless of priority, the term multiplicity has a multiplicity of well-established usage,
so one should check what usage any given author is employing.

\section{Affine space and projective space}

Let $\field$ be an algebraically closed field. For $n\geq0$, let $\Aff{n}$ denote $\field^n$, and let
$A=\field[\Aff{n}]$ denote $\field[X_1,\ldots, X_n]$. We refer to
$\Aff{n}$ as \emph{affine $n$-space}.
For any subset $S\subseteq \Aff{n}$, let $I(S)\subseteq A$ denote the ideal 
of all polynomials that vanish on $S$. (For those familiar with $\Spec$,
the \emph{affine scheme} associated to $S$ 
is $\Spec(A/I(S))$. Note that any ideal $I\subseteq A$ defines an affine subscheme
of $\Spec(A)$, and ideals $I$ and $J$ define the same affine subscheme
if and only if $I=J$.)

For $n\geq0$, let $\pr{n}$ denote equivalence classes of nonzero
$(n+1)$-tuples, where $(a_0,\ldots,a_n)$ and $(b_0,\ldots,b_n)$
are equivalent if there is a nonzero $t\in\field$ such that 
$(a_0,\ldots,a_n)=t(b_0,\ldots,b_n)$. Let $R=\field[\pr{n}]$ denote $\field[x_0,\ldots,x_n]$. 
We refer to $\pr{n}$ as \emph{projective $n$-space}.
For any subset $S\subseteq \pr{n}$, we obtain an associated homogeneous ideal 
(i.e., an ideal generated by homogeneous polynomials, also called forms)
$I(S)\subseteq R$, the ideal 
generated by all homogeneous polynomials that vanish on $S$, where we regard $R$ as being
a graded ring with each variable having degree 1 and constants having degree 0.
For those familiar with $\Proj$,
the \emph{projective scheme} associated to $S$ 
is $\Proj(R/I(S))$. If $M=(x_0,\ldots,x_n)$, any homogeneous ideal $I\subseteq M\subset R$ 
defines a subscheme $\Proj(R/I)\subseteq \Proj(R)=\pr n$, and homogeneous 
ideals $I\subseteq M$ and $J\subseteq M$ define the same subscheme
if and only if $I_t=J_t$ for $t\gg0$ (or equivalently, if and only if
$I\cap M^t=J\cap M^t$ for $t\gg0$), where $I_t$ and $J_t$ are the 
homogeneous components of the ideals of degree $t$. (Thus $I_t$ is the vector 
space span of the elements of $I$ of degree $t$. This applies in particular to $R$,
so $R_t$ is the $\field$-vector space span of the homogeneous polynomials in 
$R$ of degree $t$, and we have $I_t=R_t\cap I$.) Given a homogeneous ideal 
$I$, among all homogeneous ideals $J$ such that $I_t=J_t$ for $t\gg0$
there is a largest such ideal contained in $M$ which contains all of the others, called the 
\emph{saturation} of $I$, denoted $\operatorname{sat}(I)$. Thus given homogeneous 
ideals $I\subseteq M$ and $J\subseteq M$, we have $\Proj(R/I)=\Proj(R/J)$ if and only if
$\operatorname{sat}(I)=\operatorname{sat}(J)$. We say
an ideal is \emph{saturated} if it is equal to its saturation. Thus geometrically we are most interested in homogeneous ideals which are saturated,
since projective schemes are in bijective 
correspondence with the saturated homogeneous ideals. (Indeed,
readers uncomfortable with $\Proj$ can get by just thinking 
about saturated homogeneous ideals.)

We can regard $\Aff{n}\subset\pr{n}$ via the inclusion 
$(a_1,\ldots,a_n)\mapsto (1,a_1,\ldots,a_n)$.
We have an isomorphism of function fields 
$$\field(X_1,\ldots,X_n)=\field(\Aff{n})\cong\field(\pr{n})=\field(x_1/x_0,\ldots,x_n/x_0)$$
defined by $X_i\mapsto \frac{x_i}{x_0}$.

\begin{remark}
Some authors use $\Aff{n}$ to denote $\Spec(\field[x_1,\ldots,x_n])$.
Since we are assuming $\field$ is algebraically closed,
our usage is (by the Nullstellensatz) equivalent to taking $\Aff{n}$ to be the set of 
closed points (i.e., of points corresponding to maximal ideals)
of $\Spec(\field[x_1,\ldots,x_n])$. Likewise,
some authors use $\pr{n}$ to denote $\Proj(\field[x_0,\ldots,x_n])$.
In our definition, $\pr{n}$ denotes the set of closed points 
of $\Proj(\field[x_0,\ldots,x_n])$.
\end{remark}

As discussed in the previous section, 
we will denote the span of all polynomials of degree at most $t$ by $A_{\leq t}$.
Given an ideal $I\subseteq A$, let $I_{\leq t}$ denote $A_{\leq t}\cap I$, so
$I_{\leq t}$ is the subspace of $I$ spanned by all $f\in I$ of degree at most $t$.
Given an ideal $I\subseteq A$, the \emph{Hilbert function} of $I$ is the function $H^\leq_I$ where
$H^\leq_I(t)=\dim_\field (I_{\leq t})$; i.e., $H^\leq_I(t)$
is the $\field$-vector space dimension of the vector space spanned by all $f\in I$ with 
$\deg(f)\leq t$. The \emph{Hilbert function} of $A/I$ (or of the scheme $\Spec(A/I)$) is 
$H^\leq_{A/I}(t)=\dim_\field(A_{\leq t}/I_{\leq t})=\binom{n+t}{n}-H^\leq_I(t)$.
Given a homogeneous ideal $I\subseteq R$, the \emph{Hilbert function} $H_I$
of $I$ is the function $H_I(t)=\dim_\field (I_t)$; i.e., $H_I(t)$
is the $\field$-vector space dimension of the vector space 
spanned by all homogeneous $f\in I$ with 
$\deg(f)=t$. The \emph{Hilbert function} of $R/I$ (or of the scheme $\Proj(R/I)$) is 
$H_{R/I}(t)=\dim_\field(R_t/I_t)=\binom{t+n}{n}-H_I(t)$.

It is known that $H^\leq_{I}$ and $H^\leq_{A/I}$ become polynomials for $t\gg0$
(see Exercise \ref{dimbound} for an example). This polynomial 
is called the \emph{Hilbert polynomial} of $I$ or $A/I$ respectively. 
(We will see in the next section that the Hilbert polynomial for the ideal $I$ of the fat point 
subscheme $m_1p_1+\cdots+m_rp_r$
is $\binom{t+n}{n}-\sum_i\binom{m_i+n-1}{n}$.
Similarly, $\sum_i\binom{m_i+n-1}{n}$
is the Hilbert polynomial for $A/I$.)
Likewise, if $I\subseteq R$ is a homogeneous ideal, $H_I$ and $H_{R/I}$
become polynomials for $t\gg0$, called the \emph{Hilbert polynomial} of $I$ or $R/I$
as the case may be. Note that
$H^{\leq}_I(t)=H^{\leq}_A(t)-H^{\leq}_{A/I}(t)=\binom{t+n}{n}-H^{\leq}_{A/I}(t)$
for all $t\geq 0$. Using Exercise \ref{AffVProjBij} we also see that
$H_I(t)=H_R(t)-H_{R/I}(t)=\binom{t+n}{n}-H_{R/I}(t)$
for all $t\geq 0$. 

It is a significant and often difficult problem to determine the least value $i$ such that
the Hilbert polynomial and Hilbert function become equal for all $t\geq i$.
(For an ideal of fat points, this value is sometimes called the \emph{regularity index} of $I$,
and $i+1$ in the case of an ideal of fat points is known as the \emph{Castelnuovo-Mumford regularity} 
$\operatorname{reg}(I)$ of $I$.)

{\vskip\baselineskip\noindent\Large\bf Exercises}

\setcounter{theorem}{0}

\begin{exercise}\label{AffVProjBij} 
Show that there is a bijection between the set ${\mathcal M}_{\leq t}(A)$ of monomials of degree at 
most $t$ in $A=\field[x_1,\ldots,x_n]$ and the set ${\mathcal M}_t(R)$ of 
monomials of degree exactly $t$ in $R=\field[x_0,\ldots,x_n]$ for every $t\geq0$.
(This shows that $H^{\leq}_A(t)=H_R(t)$ for all $t\geq 0$.)
\end{exercise}

\begin{exercise}\label{exaffalphapower} 
If $0\neq I\subseteq A$ is an ideal, show that $\alpha(I^m)\leq m\alpha(I)$,
but if $0\neq J\subseteq R$ is homogeneous, then $\alpha(J^m)= m\alpha(J)$.
(See Exercise \ref{exaffalpha} for an example where equality in $\alpha(I^m)\leq m\alpha(I)$
fails.)
\end{exercise}

\begin{exercise}\label{sat} 
Let $I\subseteq M\subset R$ be a homogeneous ideal. Let $P$ be the ideal
generated by all homogeneous $f\in R$ such that $fM^i\subseteq I$ for some $i>0$. Show that 
$I\subseteq P$, that $P$ contains every homogeneous ideal $J\subseteq M$ such that 
$I_t=J_t$ for $t\gg0$, and that $I_t=P_t$ for $t\gg0$. Conclude that 
$P$ is the saturation of $I$ and that $P=\operatorname{sat}(P)$.
(In terms of colon ideals, $\operatorname{sat}(I)=\cup_{i\geq 1} I:M^i$.)
\end{exercise}

\section{Fat points in affine space}

A \emph{fat point} subscheme of affine $n$-space is the scheme corresponding to an ideal of the form
$I=\cap_{i=1}^rI(p_i)^{m_i}\subset A$ for a finite set of points $p_1,\ldots,p_r\in\Aff{n}$
and positive integers $m_i$. We denote $\Spec(A/I)$ in this case by
$m_1p_1+\cdots+m_rp_r$, and we denote the ideal $\cap_{i=1}^rI(p_i)^{m_i}$ by $I(m_1p_1+\cdots+m_rp_r)$.

Given distinct points $p_1,\ldots,p_r\in\Aff{n}$, let $I=\cap_{i=1}^rI(p_i)$;
following Waldschmidt \cite{refW} we define a constant we denote by $\gamma(I)$ as the following limit 
$$\gamma(I)=\lim_{m\to\infty}\frac{\alpha(\cap_{i=1}^r(I(p_i)^m))}{m}.$$
By Exercise \ref{exaffpower}, $\cap_{i=1}^r(I(p_i)^m)=I^m$, so 
$$\gamma(I)=\lim_{m\to\infty}\frac{\alpha(I^m)}{m},$$
but for a unified treatment, whether the points $p_i$ are in affine space or projective
space, it is better to take $$\gamma(I)=\lim_{m\to\infty}\frac{\alpha(\cap_{i=1}^r(I(p_i)^m))}{m}$$
as the definition of $\gamma(I)$.

We say the points $p_1,\ldots,p_r\in\Aff{n}$ are \emph{generic} points if the coordinates
of the points are algebraically independent over the prime field $\Pi_\field$ of $\field$.
(This is possible only if the transcendence degree of $\field$ over $\Pi_\field$
is at least $rn$.) The following problem is open for $n>1$ and $r\gg0$.

\begin{problem}\label{NagataConj}
Let $I$ be the ideal of $r$ generic points of $\Aff{n}$. Determine $\gamma(I)$. 
\end{problem}

There is a conjectural solution to the problem above, when $r\gg0$, due to 
Nagata \cite{refN} for $n=2$ and Iarrobino \cite{refI} for $n>2$:

\begin{conjecture}[Nagata/Iarrobino Conjecture]\label{NagIarroConj}
Let $I$ be the ideal of $r\gg0$ generic points of $\Aff{n}$. 
Then $\gamma(I)=\sqrt[n]{r}$ for $r\gg0$.
\end{conjecture}

\begin{remark}\label{valuesofgamma}
The value of $\gamma(I)$ is known for $r$ generic points of $\Aff{2}$ for $1\leq r\leq 9$ (see for example
\cite[Appendix 1]{refCh} and \cite[Theorem 7]{refN2}) or when $r$ is a square \cite{refN}.
In particular, $\gamma(I)=1$ if $r=1,2$, while $\gamma(I)=3/2$ if $r=3$,
$\gamma(I)=2$ if $r=4,5$, $\gamma(I)=12/5$ if $r=6$, 
$\gamma(I)=21/8$ if $r=7$,
$\gamma(I)=48/17$ if $r=8$, and
$\gamma(I)=\sqrt{r}$ if $r\geq 9$ is a square.
Moreover, when $n>2$ and $\sqrt[n]{r}$ is an integer,
then again $\gamma(I)=\sqrt[n]{r}$ (see \cite[Theorem 6]{refEvain}).
\end{remark}

We will for now just verify that the values given
in Remark \ref{valuesofgamma} are upper bounds.
By Exercise \ref{affstarsandbars}, the Hilbert polynomial of the ideal of
a fat point subscheme $m_1p_1+\cdots+m_rp_r\subset \Aff{n}$ 
is $\binom{t+n}{n}-\sum_i\binom{m_i+n-1}{n}$,
and so $\sum_i\binom{m_i+n-1}{n}$
is the \emph{Hilbert polynomial} for $A/I$ or equivalently for 
the scheme $m_1p_1+\cdots+m_rp_r$.

\begin{proposition}\label{valsofgammaProp}
Consider the ideal $I$ of $r$ distinct points of $\Aff{n}$.
Then $\gamma(I)\leq\sqrt[n]{r}$. Moreover, when $n=2$, we have:
$\gamma(I)=1$ if $r=1,2$; $\gamma(I)\leq 3/2$ if $r=3$;
$\gamma(I)\leq 2$ if $r=4,5$; $\gamma(I)\leq12/5$ if $r=6$; 
$\gamma(I)\leq21/8$ if $r=7$; and
$\gamma(I)\leq48/17$ if $r=8$.
\end{proposition}

\begin{proof}
For $\gamma(I)\leq\sqrt[n]{r}$, see Exercise \ref{NagataBound}. Now let $n=2$.
Say $r=1$. Then by Exercise \ref{dimbound1pt}, $H^\leq_{I^m}(t)=0$ for $t<m$ (so $\alpha(I^m)\geq m$)
and clearly $I^m$ has elements of degree $m$ (so $\alpha(I^m)\leq m$), hence
$\alpha(I^m)=m$. Thus $\gamma(I)=1$ by definition.

Now let $r=2$; let $p_1$ and $p_2$ be the $r=2$ points. Then $I^m\subseteq I(p_1)^m$,
so $\alpha(I(p_1)^m)\leq \alpha(I^m)$, hence
$1=\gamma(I(p_1))\leq \gamma(I^m)$, but again $I^m$ clearly has elements of degree $m$
(take the $m$th power of the linear polynomial defining the line through $p_1$ and $p_2$),
so $\alpha(I^m)\leq m$, hence $\gamma(I)\leq 1$ so we have $\gamma(I)=1$.

Now let $r=3$. If the points $p_1,p_2,p_3$ 
are colinear, then as for two points we have $\gamma(I)=1$. Otherwise, consider the cubic polynomial $L_{12}L_{13}L_{23}$ defining the union 
of the three lines $L_{ij}$ through pairs $\{p_i,p_j\}$ of the $r=3$ points.
But $L_{ij}\in I(p_i)\cap I(p_j)$ and $I(p_i)\cap I(p_j)=I(p_i)I(p_j)$ by Exercise
\ref{exaffpower}, so $L_{12}L_{13}L_{23}\in 
(I(p_1)I(p_2))(I(p_1)I(p_3))(I(p_2)I(p_3))=(I(p_1)I(p_2)I(p_3))^2$,
which (again by Exercise \ref{exaffpower}) is $I^2$.
Thus $L_{12}L_{13}L_{23}$ is in $I^2$ and has degree 3, so Exercise \ref{exWaldschmidt}(c) 
shows that $\gamma(I)\leq \alpha(I^2)/2 \leq 3/2$.

For $r=4$, it's easy to see that $\alpha(I)\leq 2$, so $\gamma(I)\leq \alpha(I)/1 \leq 2$.

For $r=5$, $H^\leq_I(2)\geq \binom{2+2}{2}-5\binom{1+2-1}{2}=1$, so
$\alpha(I)\leq 2$ and $\gamma(I)\leq \alpha(I)/1 \leq 2$.

For $r=6$, through every subset of 5 of the 6 points there is (as we just saw) a conic,
hence $I^5$ contains a nonzero polynomial of degree 12 (coming from the
conics through the 6 subsets of 5 of the 6 points), so 
$\alpha(I^5)\leq 12$ and $\gamma(I)\leq \alpha(I^5)/5 \leq 12/5$.

For $r=7$, there is a cubic which has a point of multiplicity at least 2 at any one of the points
and multiplicity at least 1 at the other 6 points, since
$H^\leq_I(3)\geq \binom{3+2}{2}-\binom{2+2-1}{2}-6\binom{1+2-1}{2}=1$.
Multiplying together the seven cubics (one having a point of multiplicity at least 2 at the first point, the next
having a point of multiplicity 2 at the second point, etc.) gives a polynomial
of degree 21 having multiplicity at least 8 at each of the points,
so $\gamma(I)\leq \alpha(I^8)/8 \leq 21/8$.

For $r=8$, there is a sextic which has a point of multiplicity at least 3 at any one of the points
and multiplicity at least 2 at the other 7 points, since
$H^\leq_I(6)\geq \binom{6+2}{2}-\binom{3+2-1}{2}-7\binom{2+2-1}{2}=1$.
Multiplying together the eight sextics gives a polynomial
of degree 48 having multiplicity at least 17 at each of the points,
so $\gamma(I)\leq \alpha(I^{17})/17 \leq 48/17$.
\end{proof}

We will see in Section \ref{BezoutSection} and its exercises 
and Section \ref{Weyl}
why equality holds above for $r<9$ when $n=2$ if the points are 
sufficiently general.

{\vskip\baselineskip\noindent\Large\bf Exercises}

\setcounter{theorem}{0}

\begin{exercise}\label{exaffpower} Let $p_1,\ldots,p_r$ be distinct points of $\Aff{n}$.
Show that $\cap_{i=1}^rI(p_i)^{m_i}=I(p_1)^{m_1}\cdots I(p_r)^{m_r}$.
\end{exercise}

\begin{exercise}\label{exWaldschmidt}[Waldschmidt's constant, \cite{refW,refW2}] 
Let $p_1,\ldots,p_r$ be distinct points of $\Aff{n}$
and let $I=\cap_{i=1}^rI(p_i)$. Let $b$ and $c$ be positive integers.
\begin{description}
\item{(a)} Show that 
$$\frac{\alpha(I^{bc})}{bc}\leq \frac{\alpha(I^{b})}{b}.$$
\item{(b)} Show that 
$$\lim_{m\to\infty}\frac{\alpha(I^{m!})}{m!}$$
exists.
\item{(c)} Show that 
$$\lim_{m\to\infty}\frac{\alpha(I^{m})}{m}$$
exists, is equal to the limit given in (b) and satisfies
$$\lim_{m\to\infty}\frac{\alpha(I^{m})}{m}\leq \frac{\alpha(I^t)}{t}$$
for all $t\geq 1$.
\end{description}
\end{exercise}

\begin{exercise}\label{affstarsandbars}
Show that the $\field$-vector space dimension of $A_{\leq t}$ is
$\dim_\field(A_{\leq t})=\binom{t+n}{n}$.
\end{exercise}

\begin{exercise}\label{starsandbars}
Show that there are $\binom{t+n}{n}$ monomials of degree $t$ in $n+1$ variables.
\end{exercise}

\begin{exercise}\label{dimbound1pt}
Let $I$ be the ideal of the point $p=(a_1,\ldots,a_n)\in\Aff{n}$. Show that 
$H^\leq_{I^m}(t)\geq\binom{t+n}{n}-\binom{m+n-1}{n}$,
with equality for $t\geq m-1$.
\end{exercise}

\begin{exercise}\label{exaffalpha} Let $p_1,p_2,p_3$ be distinct noncolinear points of $\Aff{2}$.
If $I=I(p_1)\cap I(p_2)$, show that $\alpha(I^m)=m\alpha(I)$.
If $J=I(p_1)\cap I(p_2)\cap I(p_3)$ and $m>1$, show that 
$\alpha(J^m)<m\alpha(J)$.
\end{exercise}

\begin{exercise}\label{nondecrAff}
Let $I\subseteq A$ be an ideal.
Show that $H^{\leq}_{A/I}$ is nondecreasing.
\end{exercise}

The following exercise is a version of the Chinese Remainder Theorem.

\begin{exercise}\label{dimbound}
Let $I$ be the ideal of $m_1p_1+\cdots+m_rp_r$ for $r$ distinct points $p_i\in\Aff{n}$. 
Show that $H^\leq_{I}(t)\geq\binom{t+n}{n}-\sum_i\binom{m_i+n-1}{n}$,
with equality if $t\gg0$.
\end{exercise}

\begin{exercise}\label{NagataBound}
Let $I$ be the ideal of $r$ distinct points of $\Aff{n}$. Show that $\gamma(I)\leq \sqrt[n]{r}$. 
If $1\leq r\leq n$, show that $\gamma(I)=1$.
\end{exercise}

\begin{exercise}\label{NagConj}
If $s\geq9$ and $n=2$, 
show that $\inf\{\frac{t}{m}: \binom{t+n}{n}-s\binom{m+n-1}{n}>0; m,t\geq1\}=\sqrt[n]{s}$.
(The same fact is true for $n>2$ with $s\gg0$ replacing $s\ge 9$.
This is part of the motivation for the Conjecture \ref{NagIarroConj}.)
\end{exercise}

\begin{exercise}\label{units}
Let $p\in\Aff{n}$ and let $m>0$. Show that every element $\overline{f}\in A/(I(p))^m$ is the image of
a polynomial $f\in A$ of degree at most $m-1$, and that $\overline{f}$ is a unit if and only if $f(p)\neq0$.
\end{exercise}

\begin{exercise}\label{AffIsIrr}
For any nonzero element $f\in\field[\Aff{n}]$, show there exists a point $p\in\Aff{n}$ such that
$f(p)\neq0$.
\end{exercise}

\begin{exercise}\label{pointsavoidance}
Let $n\geq 1$ and let $p_1,\ldots,p_r$ be distinct points of $\Aff{n}$. Show that there is 
a linear form $f\in\field[\Aff{n}]$ such that $f(p_i)\neq f(p_j)$ whenever $p_i\neq p_j$.
\end{exercise}

Here is a more explicit version of Exercise \ref{dimbound}, one solution
of which applies Exercises \ref{units}, \ref{AffIsIrr} and \ref{pointsavoidance}. 

\begin{exercise}\label{dimbound2}
Let $I$ be the ideal of $m_1p_1+\cdots+m_rp_r$ for $r$ distinct points $p_i\in\Aff{n}$. 
Show that $H^\leq_{I}(t)=\binom{t+n}{n}-\sum_i\binom{m_i+n-1}{n}$ if $t\geq m_1+\cdots+m_r-1$.
If the points are colinear, show that 
$H^\leq_{I}(t)>\binom{t+n}{n}-\sum_i\binom{m_i+n-1}{n}$ if $t<m_1+\cdots+m_r-1$.
\end{exercise}

\section{Fat points in projective space}

A \emph{fat point} subscheme of projective $n$-space is the scheme corresponding to an ideal of the form
$I=\cap_{i=1}^rI(p_i)^{m_i}\subset R$ for a finite set of distinct points $p_1,\ldots,p_r\in\pr{n}$
and positive integers $m_i$. We again denote the subscheme defined by $I$ by
$m_1p_1+\cdots+m_rp_r$ (in this case the subscheme is $\Proj(R/I)$), and 
we denote the ideal $\cap_{i=1}^rI(p_i)^{m_i}$ by $I(m_1p_1+\cdots+m_rp_r)$.

\begin{remark}
If $p_1,\ldots,p_r\subset \Aff{n}\subset\pr{n}$, then there is no ambiguity in the notation
$m_1p_1+\cdots+m_rp_r$, since there is a canonical isomorphism
from $m_1p_1+\cdots+m_rp_r$ regarded as a subscheme of $\Aff{n}$ and
$m_1p_1+\cdots+m_rp_r$ regarded as a subscheme of $\pr{n}$.
However, there is ambiguity in the notation $I(m_1p_1+\cdots+m_rp_r)$,
so we will sometimes use $I_A(m_1p_1+\cdots+m_rp_r)$ to denote the ideal in $A$
and $I_R(m_1p_1+\cdots+m_rp_r)$ to denote the homogeneous ideal in $R$
of $m_1p_1+\cdots+m_rp_r$. 
\end{remark}

\begin{remark}\label{sympowerRem}
If $I_R=\cap_{i=1}^rI_R(p_i)$, it can sometimes happen that 
$I_R^m=\cap_{i=1}^r(I_R(p_i)^m)$, but 
$I_R(p_1)^{m_1}\cdots I_R(p_r)^{m_r}= \cap_{i=1}^rI_R(p_i)^{m_i}$
essentially never happens (see Exercise \ref{homogprod}), and
in general the most one can say about $I^m_R$ is that
$I_R^m\subseteq \cap_{i=1}^r(I_R(p_i)^m)$.
Thus, we define the $m$th \emph{symbolic} power $I_R^{(m)}$ of $I_R=\cap_{i=1}^rI_R(p_i)$
to be $I_R^{(m)}=\cap_{i=1}^r(I_R(p_i)^m)$.
One can see the difference between the ideals
$I_R(p_1)^{m_1}\cdots I_R(p_r)^{m_r}$ and $\cap_{i=1}^rI_R(p_i)^{m_i}$
and between $I_R^m$ and $I_R^{(m)}$ by looking at primary decompositions.
The intersection $\cap_{i=1}^r(I_R(p_i)^m)$ is the primary decomposition
of $I_R^{(m)}$, but $I_R^m$ has a primary decomposition of the form
$I_R^{(m)}\cap J$ where $J$ is $M$-primary (possibly $J=M$, in which case
we have $I_R^m=I_R^{(m)}\cap M=I_R^{(m)}$), $M$ being the irrelevant ideal
(the ideal generated by the coordinate variables in $\field[\pr n]$).
Similarly, the primary decomposition of
$I_R(p_1)^{m_1}\cdots I_R(p_r)^{m_r}$ also has the form 
$I_R^{(m)}\cap J$ where $J$ is $M$-primary.
In any case, we see that $I_R^m\subseteq I_R^{(m)}$ for all $m\geq 1$.
We also see that $(I_R^m)_t=(I_R^{(m)})_t$ for $t\gg0$,
since for large $t$, any $M$-primary ideal $J$ contains $M^t$ and thus has $J_t=M_t$.
\end{remark}

By Exercise \ref{easyContainment}, we have $I^r \subseteq I^{(m)}$ if and only if
$r\geq m$. However, it is a hard problem to determine for which $m$ 
and $r$ we have $I^{(m)}\subseteq I^r$.
See for example \cite{refELS,refHH,refCHT, refHaHu} and the references therein.

\begin{problem}\label{NotOpenProb}
Let $p_1,\ldots,p_s\in \pr n$ be distinct points. Let $I=I_R(p_1+\cdots+p_s)$.
Is it true that $I^{(ns-n+1)}\subseteq I^s$ for all $s\geq 1$?
In particular, is it true that $I^{(3)}\subseteq I^2$ always holds when $n=2$?
\end{problem}

\begin{remark}\label{notopenrem}
Problem \ref{NotOpenProb} was open when the course these notes are based on was given in
2012. The situation changed shortly thereafter. An example with $I^{(3)}\not\subseteq I^2$
was posted
to the arXiv early in 2013 \cite{refDST}. This inspired another example
\cite[Remark 3.11]{refBCH};
see also \cite{refHaSe} for further discussion. Thus the problem now seems to be to 
classify the configurations of points in the plane for which we have
$I^{(3)}\not\subseteq I^2$. So far, they seem to be quite rare.
\end{remark}

Let $\delta_t:R_t\to A_{\leq t}$ be the map defined for any $F\in R_t$
by $\delta_t(F)=F(1,X_1,\ldots,X_n)$
and let $\eta_t:A_{\leq t}\to R_t$ be the map $\eta_t(f)=x_0^tf(x_1/x_0,\ldots,x_n/x_0)$.
Note that these are $\field$-linear maps, each being the inverse of the other.
In particular, $\dim(R_t)=\dim(A_{\leq t})=\binom{t+n}{n}$.

If $p\in\Aff{n}\subset\pr{n}$, so $p=(a_1,\ldots,a_n)\in \Aff n$ and can be 
represented in projective coordinates by $p=(1,a_1,\ldots,a_n)\in\pr n$,
let $I=(X_1-a_1,\ldots,X_n-a_n)$ be the ideal of $p$ in $A$
and let $J=(x_1-a_1x_0,\ldots, x_n-a_nx_0)$ be the ideal of $p\in \pr n$
in $R$. Then $\eta_t((I^m)_{\leq t})\subseteq (J^m)_t$ and
$\delta_t((J^m)_t)\subseteq (I^m)_{\leq t}$, so 
we have $\field$-linear vector space isomorphisms
$(I^m)_{\leq t}\to(J^m)_t$ given by $\eta_t$, hence $H^\leq_{I^m}(t)=H_{J^m}(t)$
and $H^\leq_{A/I^m}(t)=H_{R/J^m}(t)$ for all $t$.
Similarly, if $p_1,\ldots,p_r\in \Aff{n}\subset\pr n$, and if
$I=I_A(m_1p_1+\cdots+m_rp_r)\subset A$ and 
$J=I_R(m_1p_1+\cdots+m_rp_r)\subset R$, then again
we have $\field$-linear isomorphisms
$I_{\leq t}\to J_t$ given by $\eta_t$, hence  
\begin{equation}\label{affprofeq}
H^\leq_I(t)=H_J(t) \text{ and } H^\leq_{A/I}(t)=H_{R/J}(t)
\end{equation}
for all $t$.
Hence the Hilbert functions and Hilbert polynomials
for $m_1p_1+\cdots+m_rp_r$ are the same whether we regard them 
as affine or projective subschemes.
In particular, if $p_1,\ldots,p_r\subset \Aff{n}\subset\pr{n}$
and if $I_A=I_A(p_1+\cdots+p_r)$ and $I_R=I_R(p_1+\cdots+p_r)$,
then $\alpha(I_R^{(m)})=\alpha(I_A^m)$ for all $m\geq 1$ and
$\gamma(I_A)=\gamma(I_R)$. By Exercise \ref{dimbound}, we also have
$H_{I_R}(t)\geq \binom{t+n}{n}-\sum_i\binom{m_i+n-1}{n}$ and hence clearly
$$H_{I_R}(t)\geq \max\Big\{\binom{t+n}{n}-\sum_i\binom{m_i+n-1}{n},0\Big\}.$$
This is an equality for $t\gg0$. 
There is a conjecture, known as the SHGH Conjecture, that gives a conjectural value
for $H_{I_R}(t)$ when $n=2$ and the points $p_i$ are generic.
Here is a simple to state special case of the SHGH Conjecture, named
for various people who published
what turns out to be equivalent conjectures:
B.\ Segre \cite{refSe} in 1961, B.\ Harbourne \cite{refVanc} in 1986, 
A.\ Gimigliano \cite{refG} in 1987 (also see \cite{refG2}) and A.\ Hirschowitz \cite{refHi} in 1989.

\begin{conjecture}[SHGH Conjecture (special case)]\label{SHGHconjSpecCase}
Given $r\geq9$ generic points $p_i\in\pr2$ and any nonnegative integers $m$ and $t$,
let $I=I_R(m(p_1+\cdots+p_r))$. Then
$$H_{I}(t)=\max\Big\{\binom{t+2}{2}-r\binom{m+1}{2},0\Big\}.$$
\end{conjecture}

There has been a lot of work done on this conjecture (see for example
\cite{refAH, refCM, refHR}, but there are many more papers than this). 
The SHGH Conjecture is, however, only a starting point: 
one might also want to know the graded Betti 
numbers for a minimal free resolution. There are conjectures and results here too,
mostly for $\pr2$. See for example \cite{refIGC} for some conjectures, and
\cite{refBI, refC, refFHH, refGHI, refGI, refFreeRes, refHHF, refIda} for various results.

Most questions about fat points can be studied either from the point of view of subschemes of affine space
or of subschemes of projective space. It can be more convenient to work with homogeneous ideals,
so we will focus on the latter point of view.

We now mention some bounds on $\gamma(I)$ for an ideal $I=I_R(p_1+\cdots+p_r)$ of 
distinct points $p_i\in\pr n$.
Waldschmidt and Skoda \cite{refW, refW2, refSk} 
showed that $\gamma(I)\geq \frac{\alpha(I^{(m)})}{m+n-1}$ holds
over the complex numbers for all positive integers $m$, and in particular
that $\gamma(I)\geq \frac{\alpha(I)}{n}$. The proof involved some hard complex analysis.
Easier and more general proofs which hold for any field $\field$ in any characteristic
can be given using recent results on containments of
symbolic powers in ordinary powers of $I$: we know by \cite{refELS, refHH} that
$I^{(nm)}\subseteq I^m$ holds for all $m\geq1$. 
Thus $m\alpha(I)=\alpha(I^m)\leq\alpha(I^{(nm)})$, so dividing by $mn$ and taking the limit as $m\to\infty$ gives
$$\frac{\alpha(I)}{n}\leq\gamma(I).$$
(See \cite{refSc} for a different specifically characteristic $p>0$ argument.)

Chudnovsky \cite{refCh} showed $\frac{\alpha(I)+1}{2}\leq\gamma(I)$ in case $n=2$ and
conjectured $\frac{\alpha(I)+n-1}{n}\leq\gamma(I)$ in general; this conjecture is still open.
By Exercise \ref{W-Sk general bound} we know 
$$\frac{\alpha(I^{(m)})}{n+m-1}\leq\gamma(I).$$
Esnault and Viehweg \cite{refEV} obtained $\frac{\alpha(I^{(m)})+1}{m+n-1}\leq\gamma(I)$ in characteristic 0.
It seems reasonable to extend Chudnovsky's conjecture \cite[Question 4.2.1]{refHaHu}:

\begin{conjecture}\label{genChud}
For an ideal $I=I_R(p_1+\cdots+p_r)$ of 
distinct points $p_i\in\pr n$ and for all $m\geq 1$, 
$$\frac{\alpha(I^{(m)})+n-1}{n+m-1}\leq\gamma(I).$$
\end{conjecture}

If this conjecture is correct, it is sharp, since there are configurations of points
(so-called star configurations) for which equality holds (apply
\cite[Lemma 8.4.7]{refB. et al} with $j=1$).

{\vskip\baselineskip\noindent\Large\bf Exercises}

\setcounter{theorem}{0}

\begin{exercise}\label{homogprod}
Given $r>1$ and distinct points $p_1,\ldots,p_r\in\pr{n}$ with $m_i>0$ for all $i$, show that
$I(p_1)^{m_1}\cdots I(p_r)^{m_r}\subsetneq \cap_{i=1}^rI(p_i)^{m_i}$.
\end{exercise}

\begin{exercise}\label{nondecrHF}
Let $p_1,\ldots,p_r\in\pr n$ be distinct points. Let 
$I=I_R=I(m_1p_1+\cdots+m_rp_r)\subset R$.
Show that multiplication by a linear form $F$ that does not vanish at any of the points $p_i$
induces injective vector space homomorphisms $R_t/I_t\to R_{t+1}/I_{t+1}$.
Conclude that $H_{R/I}$ is a nondecreasing function of $t$.
\end{exercise}

\begin{exercise}\label{incrHF}
Let $p_1,\ldots,p_r\in\pr n$ be distinct points. Let 
$I=I_R=I(m_1p_1+\cdots+m_rp_r)\subset R$.
Show that $H_{R/I}(t)$ is strictly increasing until it becomes constant (i.e.,
if $c$ is the least $t$ such that $H_{R/I}(c)=H_{R/I}(c+1)$,
show that $H_{R/I}(t)$ is a strictly increasing function for $0\leq t\leq c$, and that
$H_{R/I}(t)=H_{R/I}(c)$ for all $t\geq c$).
\end{exercise}

\begin{exercise}\label{decrHF}
Give an example of a monomial ideal $J\subset\field[x,y]$ such that
$H_{R/J}$ is eventually constant but is not nondecreasing.
\end{exercise}

\begin{exercise}\label{SHGHimpliesNag}
Show that Conjecture \ref{SHGHconjSpecCase} 
implies the $n=2$ case of Conjecture \ref{NagIarroConj}.
\end{exercise}

\begin{exercise}\label{W-Sk general bound}
If $I\subset R$ is the radical ideal of a finite set of points in $\pr n$, then 
$I^{((m-1+n)t)}\subseteq (I^{(m)})^t$ \cite{refELS, refHH}.
Use this
to show $$\frac{\alpha(I^{(m)})}{n+m-1}\leq\gamma(I).$$
\end{exercise}

\begin{exercise}\label{easyContainment}
Let $r,m\geq 1$. If $I=I(p_1+\cdots+p_s)\subset R$ is the radical ideal of a finite set of 
distinct points $p_i\in\pr n$, show $I^r\subseteq I^{(m)}$ if and only if $r\geq m$.
\end{exercise}

\section{Examples: bounds on the Hilbert function of fat point subschemes of $\pr2$}

Let $p_1,\ldots,p_r\in\pr2$ be distinct points. Let $m_1,\ldots,m_r$ be positive integers.
Let $L_0,\ldots,L_{s-1}$ be lines, repeats allowed,
such that every point $p_i$ is on at least $m_i$ of the lines $L_j$.
Let $Z_0=Z=m_1p_1+\cdots+m_rp_r$. 
Define $Z_{j+1}$, for $j=0,\ldots,s-1$, recursively as follows. 
We set $m_{i0}=m_i$ for all $i$ and $Z_j=m_{1j}p_1+\cdots+m_{rj}p_r$.
Then $Z_{j+1}=m_{1\,j+1}p_1+\cdots+m_{r\,j+1}p_r$
where $m_{i\, j+1}=m_{ij}$ if $p_i\not\in L_j$,
$m_{i\, j+1}=0$ if $m_{ij}=0$, and
$m_{i\, j+1}=m_{ij}-1$ if $p_i\in L_j$ and $m_{ij}>0$.
We get a sequence of fat point subschemes 
$Z=Z_0\supseteq Z_1 \supseteq \cdots \supseteq Z_s=\varnothing$.
Geometrically, $Z_{j+1}$ is the fat point subscheme residual  
to $Z_j$ with respect to the line $L_j$. Algebraically,
$I(Z_{j+1})=I(Z_j):(F_j)$, where $F_j$ is the form defining
the line $L_j$.

Define a reduction vector ${\bf d}=(d_0,\ldots,d_{s-1})$, where 
$d_j=\sum_{p_i\in L_j}m_{i\,j-1}$, so $d_j$
is the sum of the multiplicities $m_{i\,j-1}$ for points $p_i\in L_j$.
From the reduction vector we construct a new vector,
$\operatorname{diag}({\bf d})$. The entries of $\operatorname{diag}({\bf d})$
are obtained as follows. Make an arrangement of dots in $s$ rows,
the first row at the bottom, the next row above it (aligned at the left), and so on,
one row for each entry of ${\bf d}$, where the number of dots in each row is 
given by the corresponding entry of ${\bf d}$ and where the dots are placed at integer lattice points.
The entries of $\operatorname{diag}({\bf d})$ are obtained by counting the number of dots
on each diagonal (of slope $-1$).
Figure \figone\ is Example 2.5.5 of \cite{refCHT}, where ${\bf d}=(8,6,5,2)$ and 
$\operatorname{diag}({\bf d}) = (1, 2, 3, 4, 4, 3, 3, 1, 0, 0, \dots)$.

\begin{figure}[t]
\caption{Obtaining $\operatorname{diag}({\bf d})$ from a reduction vector ${\bf d}$.}
\setlength{\unitlength}{0.75cm}
\hspace{.15in}
\begin{picture}(6,4.5)
\multiput(0,0)(1,0){8}{\circle*{.3}}
\multiput(0,1)(1,0){6}{\circle*{.3}}
\multiput(0,2)(1,0){5}{\circle*{.3}}
\multiput(0,3)(1,0){2}{\circle*{.3}}
\put(0,0){\line(1,0){8}}
\put(0,0){\line(0,1){4}}
\put(.5,-.5){\line(-1,1){1}}
\put(1.5,-.5){\line(-1,1){2}}
\put(2.5,-.5){\line(-1,1){3}}
\put(3.5,-.5){\line(-1,1){4}}
\put(4.5,-.5){\line(-1,1){4}}
\put(5.5,-.5){\line(-1,1){3}}
\put(6.5,-.5){\line(-1,1){3}}
\put(7.5,-.5){\line(-1,1){1}}
\end{picture}
\end{figure}
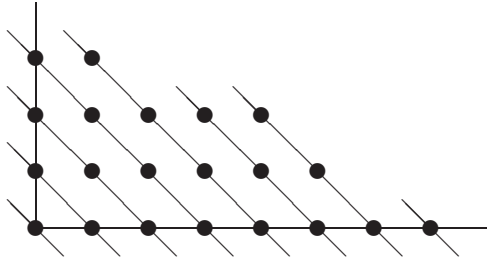

\begin{theorem}[{\cite[Theorem 1.1]{refCHT}}]\label{CHTtheorem}
Let ${\bf d}$ be the reduction vector for a fat point scheme $Z\subset \pr2$
with respect to a given choice of lines $L_i$, and let $v_{t+1}$ be the sum of the first $t+1$ entries of
$\operatorname{diag}({\bf d})$. Then $H_{R/I(Z)}(t)\geq v_{t+1}$, and equality holds for all $t$
if the entries of ${\bf d}$ are strictly decreasing.
\end{theorem}

For example, choose distinct lines $L_0,L_1,L_2$ and $L_3$.
Now choose any 8 points on $L_0$ (possibly including points of intersection 
of $L_0$ with the other lines), then any
6 additional points on $L_1$ (again possibly including points of intersection 
of $L_1$ with the other lines but avoiding points already chosen, so now we have 14
distinct points), 5 on $L_2$ (possibly including points of intersection 
of $L_2$ with the other lines but avoiding points already chosen, 
so now we have 19 distinct points)
and 2 on $L_3$ (as before possibly including points of intersection 
of $L_3$ with the other lines but avoiding points already chosen, 
so we end up with 21 distinct points). Then $Z_0$
is the reduced scheme consisting of all 21 points; removing the first 8 gives
$Z_1$, removing from $Z_1$ the next 6 gives $Z_2$, removing from
$Z_2$ the next 5 gives $Z_3$ and removing the last 2 gives $Z_4=\varnothing$.
The corresponding reduction vector is ${\bf d}=(8,6,5,2)$, and 
(regarding a function of the nonnegative integers as a sequence)
$H_{R/I(Z)}$ is $(1,3,6,10,14,17,20,21,21,21,\ldots)$.

It is sometimes convenient to give not $H_{R/I(Z)}$ itself, but its first difference $\Delta H_{R/I(Z)}$,
defined as $\Delta H_{R/I(Z)}(0)=1$ and $\Delta H_{R/I(Z)}(t)=H_{R/I(Z)}(t)-H_{R/I(Z)}(t-1)$
for $t>0$. In the preceding example, $\Delta H_{R/I(Z)}$ is $(1,2,3,4,4,3,3,1,0,0,\ldots)$.
In particular, when the entries of ${\bf d}$ are strictly decreasing, then 
$\Delta H_{R/I(Z)}=\operatorname{diag}({\bf d})$.

\begin{proof}[Sketch of the proof of Theorem \ref{CHTtheorem}]
We content ourselves here with merely obtaining an upper bound
on $H_{R/I}(t)$. The fact that this bound agrees with the statement given in the theorem
involves some combinatorial analysis, for which we refer you to the original paper.

We pause for a notational comment.
Given a line $L\subset\pr2$ and a point $p\in L\subset \pr2$, it can be ambiguous whether
by $I(p)$ we mean the ideal of $p$ in $\field[L]$ or in $\field[\pr2]$.
Thus we use $I(p)$ for the ideal in $\field[\pr2]$ and we use 
$I_L(p)$ to indicate the ideal of $p$ in $\field[L]$.

Let $Z=Z_0$ be the original fat point scheme and let $Z_1$, $Z_2$, $\ldots$, $Z_s=\varnothing$
be the successive residuals with respect to the lines $L_0,L_1,\ldots, L_{s-1}$.
Let $I=I(Z)\subset\field[\pr2]$ be the ideal defining $Z$.
Let ${\bf d}=(d_0,\ldots,d_{s-1})$. Let $F_i$ be a linear form defining $L_i$.
Given any fat point subscheme $X=a_1q_1+\cdots+a_uq_u\subsetneq\pr2$,
we have the ideal $I(X)\subset\field[\pr2]$ as usual.
Given a line $L\subset \pr2$ defined by a linear form $F$, the scheme theoretic intersection
$X\cap L=\sum_{q_i\in L}a_iq_i$ is the fat point subscheme of $L\cong\pr1$
defined by the ideal 
$I_L(X\cap L)=\cap_{q_i\in L}I_L(q_i)^{a_i}\subset\field[L]=\field[\pr2]/(F)\cong\field[\pr1]$, 
where for a point $q\in L\subset\pr2$,
$I_L(q)\subset\field[L]$ is the principal ideal defining $q$ as a point
of $L\cong\pr1$. Specifically, $I_L(q)=I(q)/(F)\subset\field[L]=\field[\pr2]/(F)$.

We have canonical inclusions $I(Z_{i+1})\to I(Z_i)$ given by multiplying by $F_i$.
The quotient $I(Z_i)/F_iI(Z_{i+1})$ is an ideal of $\field[L_i]$ whose saturation
is $I_L(Z_i\cap L_i)$. Thus we have an inclusion $I(Z_i)/F_iI(Z_{i+1})\subseteq I_{L_i}(Z_i\cap L_i)$
which need not be an equality. Thus for all $t$ we have
$I(Z_i)_t/F_i(I(Z_{i+1}))_{t-1}=(I(Z_i)/F_iI(Z_{i+1}))_t\subseteq (I_{L_i}(Z_i\cap L_i))_t$,
but for $t\gg0$ this becomes
$$I(Z_i)_t/F_i(I(Z_{i+1}))_{t-1}=(I(Z_i)/F_iI(Z_{i+1}))_t=(I_{L_i}(Z_i\cap L_i))_t.$$

Thus for each $i$ and $t$ we have an exact sequence
$$0\to (I(Z_{i+1}))_{t-1}\to (I(Z_i))_t\to (I_{L_i}(Z_i\cap L_i))_t.$$
By definition of the reduction vector, $Z_i\cap L_i$ has degree $d_i$.
Since $I_{L_i}(Z_i\cap L_i)$ is a principal ideal, we have
$\dim_\field((I_{L_i}(Z_i\cap L_i))_t)=\binom{t-d_i+1}{1}=\max\{t-d_i+1,0\}$,
since there are $t-d_i+1$ monomials in two variables of degree $t-d_i$
whenever $t-d_i\geq 0$. Thus for each $i$ we get an inequality:
for $i=0$ we have
$$\dim_\field((I(Z_0))_t)\leq\dim_\field((I(Z_1))_{t-1})+\max\{t-d_0+1,0\};$$
for $i=1$ we have
$$\dim_\field((I(Z_1))_{t-1})\leq\dim_\field((I(Z_2))_{t-2})+\max\{t-1-d_1+1,0\};$$
and continuing in this way we eventually obtain
$$\dim_\field((I(Z_{s-1}))_{t-(s-1)})\leq\dim_\field((I(Z_s))_{t-s})+\max\{t-(s-1)-d_{s-1}+1,0\}.$$
Note that 
$(I(Z_s))_{t-s}=M_{t-s}$, $M$ being the irrelevant ideal (so generated by the variables),
hence $\dim_\field((I(Z_s))_{t-s})=\binom{t-s+2}{2}$. 

By back substitution, we get 
$$\dim_\field((I(Z_0))_t)\leq\binom{t-s+2}{2}+\sum_{0\leq i\leq s-1}\max\{t-i-d_i+1,0\}.$$
Thus 
\begin{align*}
H_{R/I}(t)=&\binom{t+2}{2}-\dim_\field((I(Z_0))_t) \\
\geq&\binom{t+2}{2}-\binom{t-s+2}{2}-\sum_{0\leq i\leq s-1}\max\{t-i-d_i+1,0\}.
\end{align*}
A combinatorial analysis shows this bound is what is claimed in the statement of the theorem.
Basically, if you arrange the dots as specified by the reduction vector ${\bf d}$ (for 
Figure \figtwo, ${\bf d}=(8,5,5,2)$), then
$\binom{t+2}{2}-\binom{t-s+2}{2}-\sum_{0\leq i\leq s-1}\max\{t-i-d_i+1,0\}$ will for 
each $t$ count the number of black dots in an isosceles right triangle with legs 
of length $t$; in the Figure \figtwo\ this triangle is the big triangle, which has $t=6$. 
The term $\binom{t+2}{2}$ counts the total number of dots
in the big triangle, black and open (giving 28 in Figure \figtwo). 
To get the number of black dots, you must first
subtract the open dots in the little triangle; there are $\binom{t-s+2}{2}$ of these 
(where, in Figure \figtwo, $t=6$ and $s=4$, giving 6 open dots). The remaining terms 
then subtract off the 
number of open dots in the big triangle where each term accounts for each 
horizontal line on which there is a black dot (these terms would be
$\max\{t-0-d_0+1,0\}=\max\{6-8+1,0\}=0$ for the bottom row,
$\max\{t-1-d_1+1,0\}=\max\{6-1-5+1,0\}=1$ for the next row up,
$\max\{t-2-d_2+1,0\}=\max\{6-2-5+1,0\}=0$ for the row above that, and
$\max\{t-3-d_3+1,0\}=\max\{6-3-2+1,0\}=2$ for the top row below the little triangle).

\begin{figure}[t]
\caption{Obtaining upper bounds on Hilbert functions.}
\vskip.6in\hskip.1in
\setlength{\unitlength}{0.75cm}
\begin{picture}(6,4.5)
\multiput(0,0)(1,0){8}{\circle*{.3}}
\multiput(0,1)(1,0){5}{\circle*{.3}}
\multiput(5,1)(1,0){3}{\circle{.3}}
\multiput(0,2)(1,0){5}{\circle*{.3}}
\multiput(5,2)(1,0){3}{\circle{.3}}
\multiput(0,3)(1,0){2}{\circle*{.3}}
\multiput(2,3)(1,0){6}{\circle{.3}}
\multiput(0,4)(1,0){3}{\circle{.3}}
\multiput(0,5)(1,0){2}{\circle{.3}}
\multiput(0,6)(1,0){1}{\circle{.3}}
\put(0,0){\line(1,0){6}}
\put(0,0){\line(0,1){6}}
\put(6,0){\line(-1,1){6}}
\put(.075,4){\line(1,0){1.85}}
\put(.075,4){\line(0,1){1.9}}
\put(1.9,4){\line(-1,1){1.85}}
\end{picture}
\end{figure}
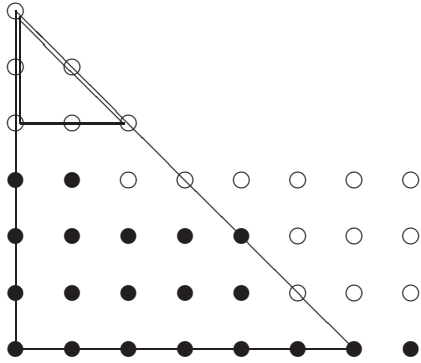

The fact that the bound is an equality when the entries of the reduction vector
are decreasing involves showing that the third map in the sequence 
$$0\to (I(Z_{i+1}))_{t-1}\to (I(Z_i))_t\to (I_{L_i}(Z_i\cap L_i))_t\eqno{(*)}$$
is surjective for every $i$ and $t$. This is done using the long exact sequence
in cohomology, where the terms in $(*)$ become modules of global sections
of ideal sheaves, and where the lack of surjectivity on the right is controlled by an $h^1$ term.
Working back from the last sequence,
one shows for each $i$ and $t$ that either the controlling $h^1$ term is 0
(and hence we have surjectivity for that $i$ and $t$)
or $(I_{L_i}(Z_i\cap L_i))_t=0$, hence again we have 
surjectivity for the given $i$ and $t$.
\end{proof}

{\vskip\baselineskip\noindent\Large\bf Exercises}

\setcounter{theorem}{0}

\begin{exercise}\label{diffOseqEx}
Let $r_1>\cdots>r_s>0$ be integers. Pick $s$ distinct lines, 
and on line $i$ pick any $r_i$ points,
such that none of the points chosen is a point of intersection 
of the $i$th line with another of the $s$ lines.
Let $Z$ be the reduced scheme consisting of all of the chosen points.
Show that $\Delta H_{R/I(Z)}$ is the sequence 
$(1,2,\ldots,s,{}^{r_s-1}s,{}^{r_{s-1}-r_s-1}(s-1),{}^{r_{s-2}-r_{s-1}-1}(s-2),\ldots)$,
where ${}^ij$ denotes a sequence consisting of $i$ repetitions of $j$.
\end{exercise}

\begin{exercise}\label{starConfigEx} 
Take any 4 distinct lines $L_0,L_1,L_2,L_3$, no three of which contain a point.
There are 6 points, $p_1,\ldots,p_6$, where pairs of the lines intersect.
Let $Z=3p_1+\cdots+3p_6$. Determine the Hilbert function of $R/I(Z)$.
(This generalizes to $s$ lines, no 3 of which are coincident at a point;
see \cite{refCHT}.)
\end{exercise}

\begin{exercise}\label{BndOnReg} 
Let $p_1,\ldots,p_r$ be distinct points of $\pr 2$.
Let $Z=m_1p_1+\cdots+m_rp_r$. 
Pick lines $L_0,\ldots,L_{r-1}$ such that $L_{i-1}$ contains $p_i$ but does not contain $p_j$
for $j\neq i$. Let ${\bf d}$ be the reduction vector obtained
by choosing $m_1$ copies of $L_0$, then $m_2$ copies of $L_1$, etc.
Show that ${\bf d}=(m_1,m_1-1,m_1-2,\ldots, m_1-(m_1-1),m_2,m_2-1,\ldots,m_2-(m_2-1),
\ldots,m_r,m_r-1,\ldots,m_r-(m_r-1))$;
conclude that $H_{R/I(Z)}(t)=\sum_i\binom{m_i+1}{2}$ for all
$t\geq m_1+\cdots+m_r-1$.
\end{exercise}

\section{Hilbert functions: some structural results}

By Exercises \ref{nondecrHF} and \ref{incrHF}, we know the Hilbert function of a fat point 
subscheme is nondecreasing in a strong way (it is strictly increasing until it is constant).
It is possible to characterize the functions that are Hilbert functions of fat point subschemes:
the Hilbert function of every fat point subscheme of projective space
is what is known as a \emph{differentiable O-sequence} (defined below), and 
for every differentiable O-sequence $f$
there is an $n$ and a finite set of points $p_1,\ldots,p_r\in\pr n$ such that
$f=H_{R/I}$ where $R=\field[\pr n]$ and $I=I_R(p_1+\cdots+p_r)$.

It is worth noting that this leads to a characterization of Hilbert functions of
reduced 0-dimensional subschemes of projective space:
a function $f$ is $H_{R/I}$ for some homogeneous radical ideal $I$ of a finite set of 
points of projective space if and only if $f$ is a 0-dimensional 
differentiable O-sequence.
It is also true that a function $f$ is $H_{R/I}$ for the homogeneous ideal 
$I=I(Z)$ for some fat point subscheme $Z$  
of projective space if and only if $f$ is a 0-dimensional 
differentiable O-sequence, but this is because reduced schemes of finite sets of points
are special cases of fat point schemes.
It is not known, for example, which 0-dimensional differentiable 
O-sequences occur
as Hilbert functions $H_{R/I^{(2)}}$ for homogeneous radical ideals
$I$ defining finite sets of points in projective space.
(A general reference for the material in this section is \cite{refBH}.)

\begin{def-prop}[see, for example, \cite{refGK}]
Let $h$ and $d$ be positive integers.  Then $h$ can be expressed uniquely in the form
$$\binom{m_d}{d} + \binom{m_{d-1}}{d-1} + \cdots + \binom{m_j}{j}$$
where $m_d > m_{d-1} > \cdots > m_j \geq j \geq 1$.  This expression for $h$ is called the \emph{$d$-binomial expansion of $h$}.  Given the $d$-binomial expansion of $h$, we also define
$$h^{\langle d \rangle} = \binom{m_d+1}{d+1} + \binom{m_{d-1}+1}{d} + \cdots + \binom{m_j+1}{j+1}.$$
\end{def-prop}

\begin{example}
The 3-binomial expansion of 15 is
$$15 = \binom{5}{3} + \binom{3}{2} + \binom{2}{1} = 10 + 3 + 2.$$
It is convenient to relate this to Pascal's triangle.
The binomial coefficients $\binom{m}{d}$ with $d$ fixed lie on a diagonal
of slope 1 say in Pascal's triangle. So to obtain the $d$-binomial expansion
of $h$, one picks the largest $\binom{m_d}{d}$ on this line less than or equal to $h$.
Then one makes up as much of the remainder $h-\binom{m}{d}$ as possible
by choosing a coefficient $\binom{m_{d-1}}{d-1}$ on the next line up of slope 1, etc.
To obtain $h^{\langle 3 \rangle}$, one just slides the choices made  for $h$
down and to the right. Thus
$$15^{\langle 3 \rangle} = \binom{6}{4} + \binom{4}{3} + \binom{3}{2} = 15 + 4 + 3 = 22.$$
\end{example}

\begin{definition}
A sequence of nonnegative integers $\{h_d\}_{d \geq 0}$ is called an \emph{O-sequence} if
\begin{itemize}
\item $h_0 = 1$, and
\item $h_{d+1} \leq h_d^{\langle d \rangle}$ for all $d \geq 1$, where $0^{\langle d \rangle}=0$ for all $i$.
\end{itemize}
\end{definition}

With these definitions we can state a well-known 
theorem of Macaulay (see \cite{refM} and \cite{refSt} for full details):

\begin{theorem}[Macaulay's Theorem]
The following are equivalent:
\begin{enumerate}
\item{(a)} $\{h_d\}_{d \geq 0}$ is an O-sequence;
\item{(b)} $\{h_d\}_{d \geq 0}$ is the Hilbert function $H_{R/I}$ for some homogeneous ideal $I \subsetneq R$; and
\item{(c)} $\{h_d\}_{d \geq 0}$ is the Hilbert function $H_{R/J}$ for some monomial ideal $J \subsetneq R$.
\end{enumerate}
\end{theorem}

\begin{definition}
Let $\mathcal{H} = \{h_d\}_{d \geq 0}$ be an O-sequence and $\Delta \mathcal{H} = \{e_d\}_{d \ge 0}$ be defined by $e_0 = h_0$ and $e_d = h_d - h_{d-1}$ for $d \geq 1$.  We say that $\mathcal{H}$ is a \emph{differentiable O-sequence} if $\Delta \mathcal{H}$ is also an O-sequence. We say $\mathcal{H}$ is \emph{0-dimensional}
if $\Delta \mathcal{H}$ is 0 for all $t\gg0$.
\end{definition}

\begin{proposition}
Let $p_1,\ldots,p_s\in\pr n$ be distinct points, let $m_1,\ldots,m_s$ be positive integers,
and let $I=I(m_1p_1+\cdots+m_sp_s)$ be the ideal of the fat point subscheme 
$m_1p_1+\cdots+m_sp_s\subset\pr n$.
Then the Hilbert function $H_{R/I}$ is a differentiable 0-dimensional O-sequence.
\end{proposition}

\begin{proof}
By Macaulay's Theorem, $H_{R/I}$ is an O-sequence. By Exercise \ref{incrHF}, $H_{R/I}$ is 
0-dimensional. But if $x\in R$ is a linear form that does not vanish at any of the points,
and if $J=I+(x)$, then 
$$\frac{R}{J}\cong\frac{R/I}{J/I}=\frac{R/I}{((x)+I)/I}\cong \frac{R/I}{x(R/I)}$$
so we have $H_{R/J}=H_{\frac{R/I}{x(R/I)}}$
and since $x$ maps to a unit in $R/I$, we obtain 
$H_{\frac{R/I}{x(R/I)}}=\Delta H_{R/I}$. But by Macaulay's Theorem again,
$H_{R/J}$ is an O-sequence, hence $H_{R/I}$ is a differentiable O-sequence.
\end{proof}

There is also a converse:

\begin{theorem}\cite{refGMR}\label{GMRthm}
Let $\mathcal{H} = \{h_d\}_{d \geq 0}$ be a differentiable 0-dimensional
O-sequence with $h_1 \leq n+1$. Then there is a finite set of points in $\pr n$
and the ideal $I \subseteq R$ of those points is a radical ideal
such that $\mathcal{H} = H_{R/I}$. In case $n=2$, those points
can be chosen as in Exercise \ref{diffOseqEx} and hence
$\Delta \mathcal{H}=\operatorname{diag}({\bf d})$ for some decreasing sequence
${\bf d}$ of positive integers.
\end{theorem}

We give some idea how one can prove this, involving
monomial ideals and their liftings. The original proof, given in \cite{refGMR},
is somewhat different.

\begin{definition}
Let $J \subseteq \field[x_1, x_2]$ be a homogeneous ideal
and let $\phi:\field[x_0, x_1, x_2]\to \field[x_1, x_2]$
be defined by $\phi(x_0)=0$ and $\phi(x_i)=x_i$ for $i>0$.
We say that $J$ \emph{lifts to $I \subseteq \field[x_0, x_1, x_2]$} if
\begin{itemize}
\item $I$ is a radical ideal in $\field[x_0, x_1, x_2]$;
\item $x_0$ is not a zero-divisor on $\field[x_0, x_1, x_2]/I$; and
\item $\phi(I)=J$.
\end{itemize} 
\end{definition}

If $\mathcal{H} = \{h_d\}_{d \geq 0}$ is a differentiable 0-dimensional O-sequence 
(with $n=2$), let $\Delta \mathcal{H} = \{e_d\}_{d \geq 0}$ be defined 
by $e_0 = 1, e_d = h_d-h_{d-1}$ for $d \geq 1$.  By Macaulay's Theorem, 
we know there exists an ideal $J \subseteq \field[x_1, x_2]$ generated by some monomials 
$\{x_1^{m_{1\,0}}x_2^{m_{2\,0}},\ldots,x_1^{m_{1r}}x_2^{m_{2r}}\}$
such that $H_{\field[x_1, x_2]/J} = \Delta \mathcal{H}$. 
Since the O-sequence is 0-dimensional, we know that
among the generators are pure powers of $x_1$ and $x_2$.
In fact, Macaulay proved more
than the statement we gave above of Macaulay's Theorem; he showed that 
$J$ can be taken to be a lex ideal,
which means that whenever $x_1^ix_2^j\in J$ with $i>0$, then 
$x_1^{i-1}x_2^{j+1}\in J$. (Here we mean lex with respect to the monomial ordering
with $x_2>x_1$, which is nonstandard, but which is needed to be consistent
with the exposition in \cite{refGGR}.)
Since in our case $J$ is not only lex but 
contains pure powers of $x_1$ and $x_2$, we may assume that 
$m_{2i}=i$ and $m_{1i}>m_{1\,i+1}$ for all $i$, with $m_{1r}=0$.
Geramita--Gregory--Roberts \cite{refGGR} and Hartshorne \cite{refHt} 
showed that $J$ lifts to an ideal $I$ which is the ideal of a 
finite set of points whose coordinates are
given by the exponent vectors $(m_{1i},m_{2i})$.
To explain this in more detail we introduce some notation and bijections.

To an element $\alpha = (a_1, a_2) \in \mathbb N^2$ we associate the 
point $\overline{\alpha} = [1:a_1:a_2] \in \mathbb P^2$.  Further, for each 
monomial $g = x^{\alpha} = x_1^{a_1}x_2^{a_2}$ we associate 
$$\overline{g} = \prod_{j=1}^2\left(\prod_{i=0}^{a_j-1}(x_j-ix_0)\right).$$
Observe that $\overline{g}$ is homogeneous.  

Now, since $J$ is a monomial ideal, the set ${\mathcal M} \setminus N$, where ${\mathcal M}$ 
denotes the monomials in $\field[x_1, x_2]$ (including 1) and $N$ denotes 
the set of monomials in $J$, gives representatives for a $\field$-basis of 
$\field[x_1, x_2]/J$.  Let $\overline{{\mathcal M}}$ denote the set of all points 
$\overline{\alpha} = \overline{(a_1, a_2)} \in \mathbb P^2$ such that 
$x_1^{a_1}x_2^{a_2} \in {\mathcal M}$.  It can then be shown (see \cite{refGGR} 
for full details) that $J$ lifts to $I = (\overline{g_i})$, where $\{g_i\}$ is 
the minimal generating set for $J$.  The key step in the proof is to show that
$$I = \{f \in \field[x_0, x_1, x_2] : f(\overline{\alpha}) = 
0 \,\,\, \mbox{for all $\overline{\alpha} \in \overline{{\mathcal M}}$}\}.$$
Note that $I$ is the ideal of a 
finite set of points which can be chosen as in Exercise \ref{diffOseqEx}.

\begin{example}
Consider $\mathcal{H} = (1, 3, 6, 9, 10, 11,11,11,\ldots)$. This is 
a differentiable 0-dimensional O-sequence
with $\Delta \mathcal{H} = (1, 2, 3, 3, 1, 1,0,0,\ldots)$.  To find a finite set of points 
$\mathbb X$ where $H_{R/I(\mathbb X)} = \mathcal{H}$ 
we consider the monomial ideal $J = (x_2^3, x_1^2x_2^2, x_1^3x_2, x_1^6)$.  
We can visualize the monomials in ${\mathcal M} \setminus N$ as the circles in the 
$x_1x_2$-plane in Figure \figthree, where the monomial $x_1^{a_1}x_2^{a_2}$ is 
represented by the pair $(a_1, a_2)$.  The open circles represent the generators of $J$.

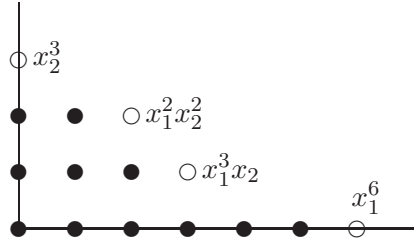
\begin{figure}[h]
\caption{A monomial ideal.}
\setlength{\unitlength}{0.75cm}
\hspace{.15in}
\begin{picture}(6,4.5)
\multiput(0,0)(1,0){6}{\circle*{.3}}
\put(0,0){\line(1,0){7}}
\put(0,0){\line(0,1){4}}
\put(0,3){\circle{.3}}
\put(.25,2.9){$x_2^3$}
\multiput(0,1)(1,0){3}{\circle*{.3}}
\put(2,2){\circle{.3}}
\put(2.25,1.9){$x_1^2x_2^2$}
\multiput(0,2)(1,0){2}{\circle*{.3}}
\put(3,1){\circle{.3}}
\put(3.25,.9){$x_1^3x_2$}
\put(6,0){\circle{.3}}
\put(5.9,.4){$x_1^6$}
\end{picture}
\end{figure}

Let $\mathbb X$ be the set consisting of the points in $\mathbb P^2$ which are in $\overline{{\mathcal M}}$; these points are
$[1:0:0]$,
$[1:1:0]$,
$[1:2:0]$,
$[1:3:0]$,
$[1:4:0]$,
$[1:5:0]$,
$[1:0:1]$,
$[1:1:1]$,
$[1:2:1]$,
$[1:0:2]$,
$[1:1:2]$.
The ideal $I = I(\mathbb X)$ is generated by:
\begin{align*}
\overline{x_2^3} & = x_2(x_2-x_0)(x_2-2x_0) \\
\overline{x_1^2x_2^2} & = x_1(x_1-x_0)x_2(x_2-x_0) \\
\overline{x_1^3x_2} & = x_1(x_1-x_0)(x_1-2x_0)x_2 \\
\overline{x_1^6} & = x_1(x_1-x_0)(x_1-2x_0)(x_1-3x_0)(x_1-4x_0)(x_1-5x_0).
\end{align*}
We have that $J$ lifts to $I$.  Observe that $\mathbb X$ is a configuration of points
contained in a union of three ``horizontal'' lines in $\pr 2$, with 6 points on the 
bottom line, 3 on the middle line and 2 on the top line.
\end{example}

The method used in the above example will work in general.  Given 
a differentiable 0-dimensional O-sequence $\mathcal{H}$ where 
$\Delta \mathcal{H} = (h_0, h_1, h_2, \ldots)$, then one applies 
the steps above using the ideal $J$ found by setting the degree $t$ 
monomials of ${\mathcal M} \setminus  N$ to be the first $h_t$ monomials in $R$ 
using lexicographic ordering.

\begin{example}
Suppose $h=(1,3,5,5,5,\ldots)$. This is a differentiable 0-dimensional O-sequence.
Using the methods of the previous section, one can check that it is the Hilbert function of
5 points in $\pr2$, 2 on one line, and three on another line, none where the lines meet.
\end{example}

\begin{example}
Suppose $h=(1,3,2,0,0\ldots)$. This is a 0-dimensional O-sequence but it is not
differentiable. It is the Hilbert function of $R/I$ for
$R=\field[x,y,z]$ and $I=(x^2,xy,x^2,y^2)+(x,y,z)^3$.
\end{example}

{\vskip\baselineskip\noindent\Large\bf Exercises}

\setcounter{theorem}{0}

\begin{exercise}\label{exO-seq1}
Let $I=I(3p)$ for a point $p\in\pr2$.
Find a set of points $p_1,\ldots,p_r\in\pr2$ such that
$H_{R/I}=H_{R/J}$ where $J=I(p_1+\cdots+p_r)$.
\end{exercise}

\begin{exercise}\label{exO-seq2}
Show that ${\bf d}$ in the statement of Theorem \ref{GMRthm}
is unique.
\end{exercise}

\section{B\'ezout's theorem in $\pr2$ and applications}\label{BezoutSection}

We start with some intuition as to what B\'ezout's Theorem is all about.
One way to think about it is as a generalization of the Fundamental Theorem of 
Algebra (FTA). One can state FTA as follows:

\begin{theorem}[FTA]\label{fta}
A nonconstant polynomial $f\in{\bf C}[x]$ of degree $d$ has exactly $d$ roots, counted with
multiplicity, where ${\bf C}$ is the field of complex numbers.
\end{theorem}

Replacing ${\bf C}$ by any algebraically closed field $\field$,
a simplified version of B\'ezout's Theorem says the following:

\begin{theorem}[Baby B\'ezout]\label{babybezout}
Let $F\in\field[\pr2]$ be a nonconstant form of degree $d$ and let $L$ be a linear form.
Then either the restriction of $F$ to $L$ has exactly $d$ roots (counted with multiplicity), 
or $L$ divides $F$. 
\end{theorem}

The full version of B\'ezout's Theorem (see below) says that forms $F,G\in \field[\pr2]$ of degrees
$d_1,d_2>0$ have exactly $d_1d_2$ common zeros (counted with multiplicity)
unless $F$ and $G$ have a common factor of positive degree.
The rigorous statement requires dealing with how to count common zeros correctly.
So let $0\neq F\in\field[\pr2]=\field[x_0,x_1,x_2]$ be homogeneous.
The multiplicity $\mult_p(F)$ of $F$ at a point $p\in\pr2$ is the largest $m$ such that
$F\in I(p)^m$, where we regard $I(p)^0$ as being $R$.
If projective coordinates are chosen so that $p=(1,0,0)$, then
$\mult_p(F)$ is the degree of a term of least degree in $F(1,x_1,x_2)$.
The homogeneous component $h$ of $F(1,x_1,x_2)$ of least degree
factors as a product of powers of homogeneous linear factors $l_i$;
i.e., $h=l_1^{m_1}\cdots l_s^{m_s}$.
The factors $l_i$ are the \emph{tangents} to $F$ at $p$,
and the exponent $m_i$ is the multiplicity of $l_i$.

Suppose $F$ and $G$ are homogeneous polynomials which do not have a common
factor vanishing at $p$. For each $m\geq1$, the $\field$-vector space dimension
of the $t$th homogeneous component of $R/((F,G)+I(p)^m)$ is
equal to some limiting value $\Lambda_m(F,G,p)$ for all $t\gg0$.
For all $m\gg0$, $\Lambda_m(F,G,p)$ also attains a limiting value, $\Lambda(F,G,p)$.
We define the \emph{intersection multiplicity} $I_p(F,G)$ to be $\Lambda(F,G,p)$.
Since $F$ and $G$ determine 1-dimensional subschemes $C_F, C_G\subset\pr2$
which in turn determine $F$ and $G$,
we also will refer to $I_p(F,G)$ as $I_p(C_F,C_G)$.

Assume that $F$, $G$ and $H$ are homogeneous polynomials which do not have a common
factor vanishing at $p$. Then some facts about intersection multiplicities are (see \cite{refHr} or \cite{refF}):
\begin{description}
\item{(a)} $I_p(F,G)\geq \mult_p(F)\mult_p(G)$, where equality holds if and only if
$F$ and $G$ have no tangent in common at $p$;
\item{(b)} $I_p(F,GH)=I_p(F,G)+I_p(F,H)$;
\item{(c)} intersection multiplicities are invariant under projective linear homogeneous changes of coordinates; and:
\end{description}

\begin{theorem}[B\'ezout's Theorem]\label{bigbezout}
If $F, G\in\field[\pr2]$ are forms which have no common factor of positive degree, then 
$$(\deg(F))(\deg(G))=\sum_{p\in\pr2}I_p(F,G).$$
\end{theorem}

\begin{example}\label{AlphaAndBezout}
Let $Z=m_1p_1+\cdots+m_rp_r$, where $p_1,\ldots,p_r\in\pr2$ 
are distinct points and each $m_i$ is a positive integer.
Let $C\subset\pr2$ be an irreducible curve of degree $d$
such that $\mult_{p_i}(C)=e_i$ for each $i$ (i.e., $\mult_{p_i}(G)=e_i$
where $G$ is the form defining $C$).
Say $0\neq F\in I(Z)_t$, so $\mult_{p_i}(F)\geq m_i$ for all $i$.
If $\sum_im_ie_i>td$, then
$\sum_iI_{p_i}(F,G)\geq \sum_i\mult_{p_i}(F)\mult_{p_i}(G)\geq\sum_im_ie_i > td$ so by B\'ezout's Theorem,
$G$ and $F$ have a common factor, but $G$ is irreducible, so
$G|F$. Thus $H\in I((m_1-e_1)p_1+\cdots+(m_r-e_r)p_r)$,
where $H=F/G$. 

We can apply this to get bounds on $\alpha(I(Z))$.
For example, let $L_1,L_2,L_3,L_4\subset\pr2$ be lines no three of which meet at a point.
We will regard $L_i$ as denoting either the line itself or the linear homogeneous form that defines the line,
depending on context.
Let $p_{ij}=L_i\cap L_j$ for $i\neq j$, so $\{p_{ij}\}$ are the six points of pair-wise intersections of the lines.
Let $Z=\sum_{ij}3p_{ij}$. It is easy to check that $(L_1L_2L_3)^2L_4$ is in $I(p_{ij})^3$ for each of the
six points. Thus $(L_1L_2L_3)^2L_4\in I(Z)_7$ so $\alpha(I(Z))\leq 7$.
On the other hand, assume we have $0\neq F\in I(Z)_6$.
There are three points where both $F$ and $L_i$ vanish, with $F$ 
having multiplicity at least 3 at each and $L_i$ having multiplicity 1.
Since $3\cdot (3\cdot 1)>\deg(F)\deg(L_i)=6$, then $L_i|F$. This is true for all
$i$, so $L_1L_2L_3L_4|F$. Let $H=F/(L_1L_2L_3L_4)$.
Then $\deg(H)=2$ and $\mult_{p_{ij}}(H)\geq 1$.
Now $3\cdot (1\cdot 1)>\deg(H)\deg(L_i)=2$, so again
$L_1\cdots L_4|H$, but this is impossible since $\deg(H)<\deg(L_1L_2L_3L_4)$.
Thus $H$ and therefore $F$ must be 0,
so $\alpha(I(Z))>6$ and hence $\alpha(I(Z))=7$.
(Note that this is in agreement with the result of Exercise \ref{starConfigEx}.)
\end{example}

\begin{example}
Let $I=I(p_1+p_2+p_3)$ for three noncolinear points of $\pr2$.
We show that $\gamma(I)=3/2$. Consider $I^{(m)}=I(m(p_1+p_2+p_3))$.
Assume $m=2s$ is even, and suppose $0\neq F\in (I^{(m)})_{3s-1}$.
Note that $F$ vanishes to order at least $m$ at each of two points
for any line $L_{ij}$ through two of the points $p_i, p_j$, $i\neq j$.
Since $2m=4s>3s-1$, this means by B\'ezout that the linear forms (also denoted $L_{ij}$)
defining the lines are factors of
$F$. Dividing $F$ by $L_{12}L_{13}L_{23}$ we obtain a form $G$ of degree $3(s-1)-1$
in $I^{(m-2)}$. The same argument applies: $L_{12}L_{13}L_{23}$ must divide $G$.
Eventually we obtain a form of degree 2 divisible by $L_{12}L_{13}L_{23}$, which is impossible.
Thus $F=0$, and $\alpha(I^{(m)})> \frac{3m}{2}-1$. Since $(L_{12}L_{13}L_{23})^s\in I^{(m)}$,
we see that $\alpha(I^{(m)})\leq \frac{3m}{2}$, thus $\alpha(I^{(m)})=\frac{3m}{2}$, and hence
$\gamma(I)=\lim_{m\to\infty}\alpha(I^{(m)})/m=3/2$.
\end{example}

{\vskip\baselineskip\noindent\Large\bf Exercises}

\setcounter{theorem}{0}

\begin{exercise}\label{intmultEx1}
Show that $I_p(F,G)=0$ if either $F$ or $G$ does not vanish at $p$.
\end{exercise}

\begin{exercise}\label{intmultEx2}
Let $p=(1,0,0)$, $F=x_1x_0-x_2^2$ and $G=x_1x_0^2-x_2^3$. Compute $I_p(F,G)$
and verify that $\sum_{p\in\pr2}I_p(F,G)=\deg(F)\deg(G)$ by explicit computation.
\end{exercise}

\begin{exercise}\label{intmultEx3}
Consider the $\binom{s}{2}$ points of pairwise intersection of $s$ distinct lines in $\pr2$,
no three of which meet at a point. Let $I$ be the radical ideal of the points.
Mimic Example \ref{AlphaAndBezout} to 
show that $\alpha(I^{(m)})=ms/2$ if $m$ is even, and $\alpha(I^{(m)})=(m+1)s/2-1$ if $m$ is odd.
\end{exercise}

\begin{exercise}\label{gammafor4pts}
Let $I=I(p_1+p_2+p_3+p_4)$ for four points of $\pr2$, no three of which are colinear.
Show that $\gamma(I)=2$. 
\end{exercise}

\begin{exercise}\label{gammafor5pts}
Let $I=I(p_1+p_2+p_3+p_4+p_5)$ for five points of $\pr2$, no three of which are colinear.
Show that $\gamma(I)=2$. 
\end{exercise}

\begin{exercise}\label{6ptsNotOnaConic}
Show that there exist 6 points of $\pr2$ which do not all lie on any conic, and no three of which are colinear.
\end{exercise}

\begin{exercise}\label{gammafor6pts}
Let $I=I(p_1+\cdots+p_6)$ for six points of $\pr2$, no three of which are colinear and which do not all
lie on a conic (such point sets exist by Exercise \ref{6ptsNotOnaConic}).
Show that $\gamma(I)=12/5$. 
\end{exercise}

\begin{exercise}\label{7generalpts}
Show that there exist 7 points of $\pr2$ no three of which are colinear and no six of which
lie on a conic.
\end{exercise}

\begin{exercise}\label{gammafor7pts}
Let $I=I(p_1+\cdots+p_7)$ for seven points of $\pr2$, no three of which are colinear and no six of which
lie on a conic (such point sets exist by Exercise \ref{7generalpts}). Show that $\gamma(I)=21/8$. 
\end{exercise}

\begin{exercise}\label{9generalpts}
Given 9 distinct points $p_i\in\pr2$ on an irreducible cubic $C$ such that $\mult_{p_i}(C)=1$
for all $i$, show that $\gamma(I)=3$ for $I=I(p_1+\cdots+p_9)$. 
\end{exercise}

\section{Divisors, global sections, the divisor class group and fat points}\label{Weyl}

For this section, our references are \cite{refHr}, \cite{refN2}, \cite{refdv},
\cite{refKrakow} and \cite{refDuke}.
Given any finite set of distinct points $p_1,\ldots,p_r\in\pr2$, there is a projective algebraic surface
$X$, a projective morphism $\pi:X\to\pr2$ (obtained by blowing up the points $p_i$)
such that each $\pi^{-1}(p_i)=E_i$ is a smooth rational curve
and such that $\pi$ induces an isomorphism $X\setminus\cup_iE_i\to \pr2\setminus\{p_1,\ldots,p_r\}$.

The divisor class group $\operatorname{Cl}(X)$ (of divisors modulo linear equivalence, where a divisor
is an element of the free abelian group on the irreducible curves on $X$) is the free group with basis
$e_0,e_1,\ldots,e_r$, where $e_0=[E_0]$ is the class of the pullback $E_0$ to $X$ of a line $L\subset\pr2$,
and $e_i=[E_i]$ for $i>0$ is the class of the curve $E_i$. The group $\operatorname{Cl}(X)$ comes
with a bilinear form, called the intersection form, defined as $-e_0^2=e_i^2=-1$ for all $i>0$, and 
$e_i\cdot e_j=0$ for $i\neq j$. An important element, known as the canonical class,
is $K_X=-3e_0+e_1+\cdots+e_r$. If $C$ and $D$ are divisors, we define $C\cdot D=[C]\cdot [D]$.
If $C$ and $D$ are prime divisors meeting transversely, then $C\cdot D$ is just the number of points of intersection
of $C$ with $D$.

If $D$ is a divisor on $X$, its class can be written as $[D]=de_0-\sum_im_ie_i$ for some integers $d$ and $m_i$.
Associated to $D$ is an invertible sheaf $\OO_X(D)$. The space of global sections of this sheaf 
is a finite dimensional $\field$-vector space, denoted $\Gamma(\OO_X(D))$ and also $H^0(X,\OO_X(D))$.
The dimension of this vector space is denoted $h^0(X,\OO_X(D))$; if $[D]=[D']$, then
$h^0(X,\OO_X(D))=h^0(X,\OO_X(D'))$.

In case $D=dE_0-\sum_im_iE_i$ such that each $m_i\geq0$, then there is a canonical identification of
$H^0(X,\OO_X(D))$ with $I(m_1p_1+\cdots+m_rp_r)_d$ \cite[Proposition IV.1.1]{refKrakow}. 
Thus techniques for computing
$h^0(X,\OO_X(D))$ can be applied to computing the Hilbert function of $m_1p_1+\cdots+m_rp_r$.
One important tool is the theorem of Riemann-Roch for surfaces; see Exercise \ref{RRexercise}.
B\'ezout's Theorem also has a natural interpretation in this context. If $C$ and $D$ are effective divisors
such that $[C]=c_0e_0-c_1e_1-\cdots-c_re_r$ and $[D]=d_0e_0-d_1e_1-\cdots-d_re_r$,
then $C\cdot D=c_0d_0-c_1d_1-\cdots-c_rd_r$; if this is negative then $C$ and $D$ have a common 
component. In particular, if $C$ is a prime divisor, then $C$ itself is the common component,
hence $D-C$ is effective.

Another important technique involves a group action on $\operatorname{Cl}(X)$
related to the Cremona group of birational transformations of the plane. Given $\pi:X\to\pr2$ as above,
there can exist morphisms $\pi':X\to\pr2$ obtained by blowing up other points (possibly infinitely near)
$p_1',\ldots,p_r'\in\pr2$. The composition $\pi'\pi^{-1}$, defined away from the points $p_i$,
is a birational transformation of $\pr2$, hence an element of the Cremona group
(named for Luigi Cremona, after whom there is named a street in Rome near the Colosseum).
We thus have a second basis $e_0',e_1',\ldots,e_r'$ of $\operatorname{Cl}(X)$
corresponding to curves $E_i'$. In particular, we can write $dE_0-\sum_im_iE_i$ as $d'E_0'-\sum_im_i'E_i'$. 
The change of basis transformation from the basis $e_i$ to the basis $e_i'$ is always an element
of a particular group, now known as the Weyl group, $W_r$
(we give generators $s_i$ for $W_r$ below). 
For $r<9$, $W_r$ is finite, but it is infinite for all $r\geq9$. 

\begin{example}
Consider the quadratic Cremona transformation on $\pr2$, defined away from
$x_0x_1x_2=0$ as $Q:(a,b,c)\mapsto(1/a,1/b,1/c)$. Alternatively, one can define it
at all points of $\pr2$ except $(1,0,0)$, $(0,1,0)$ and $(0,0,1)$ as
$(a,b,c)\mapsto (bc,ac,ab)$. It can also be obtained by as $\pi'\pi^{-1}$, where
$\pi:X\to\pr2$ is the morphism given by blowing up the points $(1,0,0)$, $(0,1,0)$ and $(0,0,1)$
and $\pi':X\to \pr2$ contracts the proper transforms of the lines through pairs of those points.
More generally one can define the quadratic transform at any three noncolinear points,
by blowing them up and blowing down the proper transforms of the lines through
pairs of the 3 points. An important theorem announced by M.\ Noether 
(but whose proof was felt to be incomplete), is that the Cremona group for $\pr2$ is generated
by invertible linear transformations of the plane and quadratic transformations \cite{refA}.
\end{example}

Let $n_0=e_0-e_1-e_2-e_3$ and let $n_i=e_1-e_{i+1}$ for $i=1,\ldots,r-1$.
For any $x\in\operatorname{Cl}(X)$ and any $0\leq i<r$, let $s_i(x)=x+(x\cdot n_i)n_i$.
Then $W_r$ is defined to be the group generated by $s_i\in W_r$.
When $i>0$, the element $s_i$ just transposes $e_i$ and $e_{i+1}$, so
$\{s_1,\ldots,s_{r-1}\}$ generates the group of permutations 
on the set $\{e_1,\ldots,e_r\}$. When the points
$p_1,p_2,p_3$ are not colinear,
the element $s_0$ corresponds to the quadratic transformation
$Q:(a,b,c)\mapsto(\frac{1}{a},\frac{1}{b},\frac{1}{c})$. 
Note that $s_0(e_1)=e_0-e_2-e_3$, $s_0(e_2)=e_0-e_1-e_3$, and $s_0(e_3)=e_0-e_1-e_2$:
blowing up $p_1$, $p_2$ and $p_3$, to get $E_1,E_2,E_3$ and blowing down the
proper transforms of the line through $p_2$ and $p_3$, the line through $p_1$ and $p_3$
and the line through $p_1$ and $p_2$ is precisely $Q$. 
(Note also that $s_0(e_0)=2e_0-e_1-e_2-e_3$
and a line $a_0x_0+a_1x_1+a_2x_2=0$ pulls back under $Q$ to 
$a_0/x_0+a_1/x_1+a_2/x_2=0$ which, by multiplying through by $x_0x_1x_2$ 
to clear the denominators is the same as $a_0x_1x_2+a_1x_0x_2+a_2x_0x_1=0$;
i.e., on the surface $X$ obtained by blowing up the coordinate vertices
we have $e_0'=2e_0-e_1-e_2-e_3$.)

When the points $p_i$ are sufficiently general (such as being generic,
meaning, say, that the projective coordinates $a_{ij}$ for each point $p_i=(a_{i0},a_{i1},a_{i2})$ are all nonzero,
and the ratios $\frac{a_{11}}{a_{10}}$, $\frac{a_{12}}{a_{10}}$, $\frac{a_{21}}{a_{20}}$,
$\frac{a_{22}}{a_{20}},\ldots,\frac{a_{r1}}{a_{r0}}$, $\frac{a_{r2}}{a_{r0}}$
are algebraically independent over the prime field of $\field$)
and given the surface $\pi:X\to\pr2$ obtained by blowing up the points $p_i$,
the birational morphisms $X\to\pr2$ (up to projective equivalence) are in one-to-one correspondence
with the elements of $W_r$. We denote by $\pi_w$ the morphism corresponding to $w$. 
The identity element $w$ corresponds to the
basis $\{e_0,e_1,\ldots,e_r\}$ obtained by blowing up the points $p_i$, and this gives $\pi$ since for $i>0$, 
$E_i$ is the unique effective divisor whose class is $e_i$. Contracting $E_r, E_{r-1},\ldots,E_1$ in order
gives $\pi$. Likewise, for any $w\in W_r$, the basis $e_i'=w(e_i)$ gives the sequence of curves $E_i'$
which must be contracted to define $\pi_w$.

Given a divisor $F=dE_0-\sum_im_iE_i$, we denote by $wF$ the divisor $d'E_0'-\sum_im_i'E_i'$ where
$w(de_0-\sum_im_ie_i)=d'e_0'-\sum_im_i'e_i'$. Since $w$ represents a change of basis,
we have $H^0(X,\OO_X(F))=H^0(X,\OO_X(wF))$ and thus $\dim I(\sum_im_ip_i)_d=\dim I(\sum_im_i'p_i')_{d'}$.
(The fact that $H^0(X,\OO_X(F))=H^0(X,\OO_X(wF))$ 
also shows that $I(\sum_im_ip_i)_d$ has an irreducible element if and only if 
$I(\sum_im_i'p_i')_{d'}$ does.)
But if the points $p_i$ are generic, so are the points $p_i'$ (up to projective equivalence), so 
$\dim I(\sum_im_ip_i)_d=\dim I(\sum_im_i'p_i)_{d'}$. 
(There is an automorphism $\phi:\field\to\field$ such that the coordinates of the points $p_i$
map to the coordinates of the points $p_i'$. This induces an invertible map 
$\Phi:I(\sum_im_i'p_i)_{d'}\to I(\sum_im_i'p_i')_{d'}$ such that if $a_i\in\field$ and $F_i\in I(\sum_im_i'p_i)_{d'}$,
then $\Phi(\sum_ia_iF_i)=\sum_i\phi(a_i)\Phi(F_i)$, from which it follows that 
$\dim I(\sum_im_i'p_i)_{d'}=\dim I(\sum_im_i'p_i)_{d'}$ and hence that 
$\dim I(\sum_im_ip_i)_d=\dim I(\sum_im_i'p_i)_{d'}$.)

\begin{example}\label{exccurves}
Let $p_1,\ldots,p_9$ be generic points of $\pr2$.
We show that 
$I(p_1+\cdots+p_5)_2$,
$I(2p_1+p_2+\cdots+p_7)_3$ and
$I(3p_1+2p_2+\cdots+2p_8)_6$
each are 1-dimensional, with basis given by an irreducible form.
In each case we have a homogeneous component of the form $I(\sum_im_ip_i)_d$.
It is enough to show that there is an element $w\in W_8$ such that
$w[F]=e_0-e_1-e_2$, where $[F]=de_0-\sum_im_ie_i$.
But $s_0(2e_0-e_1-\cdots-e_5)=e_0-e_4-e_5$ and we apply a permutation $\sigma$ to obtain
$\sigma(e_0-e_4-e_5)=e_0-e_1-e_2$. Thus $\dim I(p_1+\cdots+p_5)_2=\dim I(p_1+p_2)_1$
and since $I(p_1+p_2)_1$ clearly has an irreducible element so does $I(p_1+\cdots+p_5)_2$.
The other cases with $r<9$ are similar.
The case that $r=9$ is also similar if we show that $I(p_1+\cdots+p_9)_3$
has an irreducible element.
\end{example}

\newpage
{\vskip\baselineskip\noindent\Large\bf Exercises}

\setcounter{theorem}{0}

\begin{exercise}\label{isometry}
Let $X$ be the blow up of $\pr2$ at $r$ distinct points.
Show that $w(x)\cdot w(y)=x\cdot y$ for all $x,y\in\operatorname{Cl}(X)$ and all $w\in W_r$, and
show that $w(K_X)=K_X$ for all $w\in W_r$, where $K_X=-3e_0+e_1+\cdots+e_r$.
\end{exercise}

\begin{exercise}\label{RRexercise}
Let $X$ be the blow up of $\pr2$ at $s$ distinct points $p_i\in\pr2$.
Let $F=tE_0-m_1E_1-\cdots-m_sE_s$. The theorem of 
Riemann-Roch for surfaces says that 
$$h^0(X,\OO_X(F))-h^1(X,\OO_X(F))+h^2(X,\OO_X(F))=\frac{F^2-K_X\cdot F}{2}+1.$$
Serre duality says $h^2(X,\OO_X(F))=h^0(X,\OO_X(K_X-F))$,
and hence $h^2(X,\OO_X(F))=0$ if $t\geq0$.
Thus for $t\geq 0$ and $m_i\geq 0$ for all $i$, taking $I=I(m_1p_1+\cdots+m_sp_s)$, we have
$H_I(t)=h^0(X,\OO_X(F))=\frac{F^2-K_X\cdot F}{2}+1+h^1(X,\OO_X(F))$.
Show that
$$\frac{F^2-K_X\cdot F}{2}+1=\binom{t+2}{2}-\sum_i\binom{m_i+1}{2}.$$
Conclude that $P_I(t)=\frac{F^2-K_X\cdot F}{2}+1$ where $P_I$ is the Hilbert polynomial for $I$, 
and that $h^1(X,\OO_X(F))=H_I(t)-P_I(t)$ is the difference between
the Hilbert function and Hilbert polynomial for $I$. 
\end{exercise}

\begin{exercise}\label{homaloidalclass}
Let $p_1,\ldots,p_8$ be generic points of $\pr2$.
Show that $\alpha(I(6p_1+\cdots+6p_8))=17$.
\end{exercise}

\begin{exercise}\label{exceptionalclass}
Let $p_1,\ldots,p_r\in\pr2$ be generic points of $\pr2$.
Let $X$ be the surface obtained by blowing up the points.
Let $w\in W_r$ and let $[C]=w(e_1)$.
Show that $C$ is a smooth rational curve with $C^2=C\cdot K_X=1$.
Conclude that $((mC)^2-K_X\cdot (mC))/2+1\leq0$ for all $m>1$.
Such a curve $C$ is called an \emph{exceptional curve}.
(By \cite[Theorem 2b]{refN2}, when $r\geq3$, the set of classes of exceptional curves 
is precisely the orbit $W_r(e_1)$.)
\end{exercise}

\begin{exercise}\label{fcexercise}
Let $p_1,\ldots,p_r\in\pr2$ be distinct points of $\pr2$.
Let $X$ be the surface obtained by blowing up the points.
Let $C$ be an exceptional curve on $X$, let $D$ be an effective
divisor, let $m=-C\cdot D>0$ and let $F=D-mC$. If $m>1$,
show that $h^0(X,\OO_X(D))=h^0(X,\OO_X(F))$ (hence
$C$ is a \emph{fixed component} of $|D|=|F|+mC$ of multiplicity $m$, 
where $|D|$ is the linear system of all curves
corresponding to elements of $H^0(X,\OO_X(D))$),
and that $(D^2-K_X\cdot D)/2<(F^2-K_X\cdot F)/2$;
conclude that $h^0(X,\OO_X(D))>(D^2-K_X\cdot D)/2+1$.
\end{exercise}

\section{The SHGH Conjecture}

The SHGH Conjecture \cite{refSe,refVanc,refG,refHi} gives an explicit
conjectural value for the Hilbert function of the ideal of a fat point subscheme of $\pr2$
supported at generic (or even just sufficiently general) points. 

Consider $I_4$ where $I$ is the ideal of the fat point subscheme $3p_1+3p_2+p_3+p_4\subset\pr2$.
Let $D=4E_0-3E_1-3E_2-E_3-E_4$ and let $C=E_0-E_1-E_2$.
Note that $D\cdot C=-2$; let $F=D-2C=2E_0-E_1-\cdots-E_4$.
We know $H_I(4)=h^0(X,\OO_X(D))\geq (D^2-K_X\cdot D)/2+1=
\binom{4+2}{2}-2\binom{3+1}{2}-2\binom{1+1}{2}=1$.
But by Exercise \ref{fcexercise} we also have 
$$H_I(4)=h^0(X,\OO_X(F))\geq (F^2-K_X\cdot F)/2+1=2.$$
The occurrence of $C$ as a fixed component of $|D|$ of multiplicity more than 1 results in
a strict inequality $h^0(X,\OO_X(D))>(D^2-K_X\cdot D)/2+1$.

The SHGH Conjecture says that whenever we have a divisor $D=dE_0-m_1E_1-\cdots-m_rE_R$
with $d,m_1,\ldots,m_r\geq 0$,
(assuming that the $E_i$ were obtained by blowing up $r\geq 3$ generic points of $\pr2$)
then either $h^0(X,\OO_X(D))=\max(0,(D^2-K_X\cdot D)/2+1)$ or there is an exceptional curve
$C$ (i.e., an effective divisor whose class is an element of the $W_r$-orbit of $E_1$)
such that $C\cdot D<-1$. If $h^0(X,\OO_X(D))>0$, it is easy to find all such $C$
and subtract them off, leaving one with $F$ such that
$h^0(X,\OO_X(F))=(F^2-K_X\cdot F)/2+1$. (If $D\cdot C\geq 0$ for all $C$, one can show that 
$[D]$ can be reduced by $W_r$ to a nonnegative linear combination of the classes
$e_0$, $e_0-e_1$, $2e_0-e_1-e_2$, $3e_0-e_1-e_2-e_3, \cdots, 3e_0-e_1-\cdots-e_r$;
see \cite{refDuke}.)

The SHGH Conjecture is known to hold for $r\leq 9$.

\begin{example}
Consider the fat point subscheme $Z=13p_1+13p_2+10p_3+\cdots+10p_7$ for generic points $p_i\in\pr2$.
We determine the Hilbert function of $I=I(Z)$.
First $H_I(28)=0$. We have $H_I(28)=h^0(X,\OO_X(D))$ for
the divisor $D=28E_0-13E_1-13E_2-10E_3-\cdots-10E_7$.
But $[D]$ reduces via $W_7$ to $-2e_0+2e_4+2e_5+5e_6+5e_7$, so
$h^0(X,\OO_X(D))=h^0(X,\OO_X(D'))$, where $D'=-2E_0+2E_4+2E_5+5E_6+5E_7$.
The occurrence of a negative coefficient for $e_0$ means $h^0(X,\OO_X(D'))=0$,
hence $H_I(t)=0$ for $t<29$.
Now consider $D=29E_0-13E_1-13E_2-10E_3-\cdots-10E_7$.
Then via the action of $W_7$ we obtain $D'=4E_0-E_1-\cdots-E_5+2E_6+2E_7$.
As in Exercise \ref{fcexercise}, we can subtract off $2E_6+2E_7$ to get
$F=D-(2E_6+2E_7)=4E_0-E_1-\cdots-E_5=(E_0)+(3E_0-E_1-\cdots-E_5)$. 
Thus $F\cdot C\geq 0$ for all exceptional $C$, so by the SHGH Conjecture
$H_I(29)=h^0(X,\OO_X(D))=h^0(X,\OO_X(D'))=h^0(X,\OO_X(F))=(F^2-K_X\cdot F)/2 +1=10$.
Finally consider $D=30E_0-13E_1-13E_2-10E_3-\cdots-10E_7$.
Here we get $F=D'=12E_0-4(E_1+\cdots+E_5)-E_6-E_7=3(3E_0-E_1-\cdots-E_5)+(3E_0-E_1-\cdots-E_7)$.
Thus $D'\cdot C\geq 0$ for all exceptional $C$, so we get
$H_I(30)=h^0(X,\OO_X(D))=h^0(X,\OO_X(D'))=h^0(X,\OO_X(F))=(F^2-K_X\cdot F)/2 +1=39$.
For $t\geq 30$ and $D=tE_0-13E_1-13E_2-10E_3-\cdots-10E_7$, we have 
$D=(t-30)E_0+(30E_0-13E_1-13E_2-10E_3-\cdots-10E_7)$ so
$D\cdot C=(t-30)E_0\cdot C+C\cdot (30E_0-13E_1-13E_2-10E_3-\cdots-10E_7)\geq 0$
for all exceptional $C$, so 
$H_I(t)=h^0(X,\OO_X(D))=\max(0,(D^2-K_X\cdot D)/2 +1)$,
but $(D^2-K_X\cdot D)/2 +1$ was positive for $t=30$ and adding a nonnegative 
multiple of $E_0$ only makes it bigger so we have
$H_I(t)=h^0(X,\OO_X(D))=(D^2-K_X\cdot D)/2 +1=\binom{t+2}{2}-2\binom{13+1}{2}-5\binom{10+1}{2}$.
\end{example}

We close by relating the statement of the SHGH Conjecture given above
to the special case stated in Conjecture \ref{SHGHconjSpecCase}.
Consider $F=tE_0-m(E_1+\cdots+E_r)$, where $p_1,\ldots,p_r\in\pr2$ are $r\geq9$ 
generic points of $\pr2$, $X$ is the surface obtained by blowing up the points and $E_i$
is the exceptional curve obtained by blowing up $p_i$.
Let $I$ be the radical ideal of the points. 
Then $H_{I^{(m)}}(t)=h^0(X,\OO_X(F))$.
For simplicity, we consider only the cases $t\geq 3m\geq0$.
Then $F=-mK_X+(t-3m)E_0$ with $t-3m\geq0$. But for any exceptional curve $E$ we have
$[E]=w([E_1])$ for some $w\in W_r$, so $-K_X\cdot E=-K_X\cdot E_1=1$
by Exercise \ref{isometry}. Since $E$ is a curve on $X$, its image in
$\pr2$ has nonnegative degree, so $E_0\cdot E\geq 0$.
Thus $F\cdot E\geq m\geq0$. The SHGH Conjecture therefore asserts
$H_{I^{(m)}}(t)=h^0(X,\OO_X(F))=\max\big(0,(D^2-K_X\cdot D)/2+1\big)=
\max\Big(0,\binom{t+2}{2}-r\binom{m+1}{2}\Big)$,
as conjectured in Conjecture \ref{SHGHconjSpecCase}.

{\vskip\baselineskip\noindent\Large\bf Exercises}

\setcounter{theorem}{0}

\begin{exercise}\label{HilbFuncExercise}
Find the Hilbert function of the ideal $I$ of $Z=12p_1+10p_2+\cdots+10p_8\subset \pr2$,
assuming the points are generic.
\end{exercise}

\section{Solutions}

\noindent{\bf 2. Affine space and projective space}

\vskip\baselineskip\noindent {\bf Solution \ref{AffVProjBij}.}
Define a map ${}^*\!\!:{\mathcal M}_{\leq t}(A)\to {\mathcal M}_t(R)$
by $x_1^{m_1}\cdots x_n^{m_n}\mapsto (x_1^{m_1}\cdots {x_n^{m_n}})^*
=x_0^{m_0}x_1^{m_1}\cdots x_n^{m_n}$
where $m_0=t-(m_1+\cdots+m_n)$, and define 
a map ${}_*\!\!:{\mathcal M}_t(R)\to{\mathcal M}_{\leq t}(A)$ by evaluating $x_0$ at 1; i.e., by
$x_0^{m_0}x_1^{m_1}\cdots x_n^{m_n}\mapsto (x_0^{m_0}\cdots{x_n^{m_n}})_*
=x_1^{m_1}\cdots x_n^{m_n}$.
If $f$ is a monomial in ${\mathcal M}_{\leq t}(A)$, clearly $(f^*)_*=f$, while if $F\in {\mathcal M}_t(R)$,
then just as clearly $(F_*)^*=F$. Thus ${}^*$ and $_*$ are inverse to each other and hence are bijections.

\vskip\baselineskip\noindent {\bf Solution \ref{exaffalphapower}.}
Pick $f\in I$ of degree $\alpha(I)$. Then $f^m\in I^m$, so $\alpha(I^m)\leq \deg(f^m)=m\alpha(I)$.
Since $J$ is homogeneous, $J$ has a set of homogeneous generators $g_1,\ldots,g_r$, 
hence $J^m$ is generated by products of $m$ of the generators $g_i$ (repeats allowed),
the minimum degree of which is $m\alpha(J)$. 
But for any homogeneous elements $b_1,\ldots,b_t$ in $R$, where we assume (by reindexing if need be)
that $\deg(b_1)\leq \deg(b_2)\leq \cdots\leq \deg(b_t)$,
the ideal $(b_1,\ldots,b_t)$ is contained in $M^s$ for $s=\deg(b_1)$, where 
$M$ is the ideal generated by the variables.
Since $M^s$ is the span of the monomials of degree at least $s$, there are no elements in $M^s$
(and hence none in $(b_1,\ldots,b_t)$) of degree less than $s$. Applied to $J^m$,
we see that $J^m$ has an element of degree $m\alpha(J)$ and no nonzero elements of degree less than that,
hence $\alpha(J^m)= m\alpha(J)$.

\vskip\baselineskip\noindent {\bf Solution \ref{sat}.}
Since $fM\subseteq I$ for all $f\in I$, we have $f\in P$ for all homogeneous $f\in I$. Thus $I\subseteq P$.
If $J\subseteq M$ is any homogeneous ideal such that $J_t=I_t$ for all $t\gg0$, then for any homogeneous
$g\in J$ and for $i$ large enough we have $gM^i\in J_t=I_t$, hence $g\in P$, so $J\subseteq P$.
Thus $P$ contains every nontrivial homogeneous
ideal whose homogeneous components eventually coincide with those of $I$.
Since $P$ is finitely generated, there is an $s$ large enough 
such that $fM^s\subset I$ for every generator $f$
in a given finite set of homogeneous generators for $P$. Thus $PM^s\subseteq I$ for $s\gg0$. 
But for degrees $t\geq\omega$, where $\omega$ is the maximum degree 
in a minimal set of generators of $P$, we have $P_tM_1=P_{t+1}$,
hence $(PM^i)_t=P_t$ for all $t\geq \omega+i$. Thus $P_t=(PM^i)_t\subseteq I_t\subset P_t$
for $t\gg0$. Hence $P$ is the largest ideal among all homogeneous ideals $J$ such that
$J_t=I_t$ for $t\gg0$; i.e., $\operatorname{sat}(I)=P$.
Of course, by maximality of the saturation we always have
$P\subseteq\operatorname{sat}(P)$, but 
$(\operatorname{sat}(P))_t=P_t=I_t$ for $t\gg0$, hence
$\operatorname{sat}(P)\subseteq P$, so $P=\operatorname{sat}(P)$.

\vskip\baselineskip\noindent{\bf 3. Fat points in affine space}

\vskip\baselineskip\noindent {\bf Solution \ref{exaffpower}.} 
Clearly $I(p_1)^{m_1}\cdots I(p_r)^{m_r}\subseteq \cap_{i=1}^rI(p_i)^{m_i}$.
For the reverse inclusion, note that
not every polynomial which vanishes at $p_1$ vanishes at $p_2$,
so we can pick a polynomial $f$ such that $f(p_1)=0$ but $f(p_2)\neq 0$.
Normalizing allows us to assume $f(p_2)=1$. Let $g=1-f$. Then $f\in I(p_1)$,
$g\in I(p_2)$ and $f+g=1$. Writing $(f+g)^{m_1+m_2}$ as a linear combination
of terms of the form $\binom{i+j}{j}f^ig^j$ with $i+j=m_1+m_2$, each term is either in $I(p_1)^{m_1}$
or in $I(p_2)^{m_2}$. Thus we can write $1=F+G$ where $F$ is the sum of the terms 
in $I(p_1)^{m_1}$ and $G$ is the sum of the terms in $I(p_2)^{m_2}$. Therefore every element
$h\in I(p_1)^{m_1}\cap I(p_2)^{m_2}$ can be written $h=hF+hG\in I(p_1)^{m_1}I(p_2)^{m_2}$;
i.e., $I(p_1)^{m_1}\cap I(p_2)^{m_2}=I(p_1)^{m_1}I(p_2)^{m_2}$. Similarly,
$I(p_1)^{m_1}\cap I(p_2)^{m_2}\cap I(p_3)^{m_3}=I(p_1)^{m_1}I(p_2)^{m_2}\cap I(p_3)^{m_3}=
I(p_1)^{m_1}I(p_2)^{m_2}I(p_3)^{m_3}$. Continuing in this way, we eventually have
$I(p_1)^{m_1}\cap \cdots\cap I(p_r)^{m_r}=I(p_1)^{m_1}\cdots I(p_r)^{m_r}$.

\vskip\baselineskip\noindent {\bf Solution \ref{exWaldschmidt}.} (a) Since $I^{bc}=(I^b)^c$, we have
$\alpha(I^{bc})=\alpha((I^b)^c)\leq c\alpha(I^b)$ by Exercise \ref{exaffalphapower}. 
Now the result follows by dividing by $bc$.

(b) By (a), $\frac{\alpha(I^{m!})}{m!}$ is decreasing as $m$ increases but is always positive,
so it has a limit $L$.

(c) For any $\varepsilon>0$, we will show for $t\gg0$ that $L\leq \alpha(I^t)/t\leq L+ \varepsilon $.
For $m\gg0$ we may assume that $L\leq \alpha(I^{m!})/m!\leq L+ \varepsilon/2$.
For $t\geq m!$ we can write $t=s\cdot m!+d$ for some $0\leq d<m!$.
Then $I^{(s+1)m!}\subseteq I^t$, so $\alpha(I^t)\leq \alpha(I^{(s+1)m!})\leq (s+1)\alpha(I^{m!})$, so
\begin{align*}
L&\leq \frac{\alpha(I^{t!})}{t!}\leq \frac{\alpha(I^t)}{t}\leq \frac{(s+1)\alpha(I^{m!})}{s\cdot m!+d}\\
&=\frac{s\alpha(I^{m!})}{s\cdot m!+d}+\frac{\alpha(I^{m!})}{s\cdot m!+d}
\leq \frac{\alpha(I^{m!})}{m!} + \frac{\alpha(I^{m!})}{s\cdot m!}\leq L+\frac{\varepsilon}{2}+\frac{\alpha(I^{m!})}{s\cdot m!},
\end{align*}
but for $s\gg0$ (i.e., for $t\gg m!$), we have $\alpha(I^{m!})/(s\cdot m!)\leq \varepsilon/2$.
The fact that $\lim_{m\to\infty}\frac{\alpha(I^{m})}{m}\leq \frac{\alpha(I^t)}{t}$ for all $t\geq1$ follows
from (a) and (b).

\vskip\baselineskip\noindent {\bf Solution \ref{affstarsandbars}.} The vector space $A_{\leq t}$ has basis consisting of monomials $\mu$
of degree at most $t$ in the $n$ variables $X_1,\ldots,X_n$. By introducing an extra variable $X_0$,
we can create a bijection between the monomials of degree $t$ in $X_0,\ldots,X_n$ and the
monomials $\mu$ of degree at most $t$ in $X_1,\ldots,X_n$ (given by multiplying
each such $\mu$ by $X_0^i$ where $i=t-\deg(\mu)$). Now see Exercise \ref{starsandbars}.

\vskip\baselineskip\noindent {\bf Solution \ref{starsandbars}.} 
We must count the number of arrangements of $n$ ones and $t$ zeros,
since such arrangements are in bijection with the monomials in $n+1$ variables
of degree $t$ (for example, $001011$ is the monomial $x_0^2x_1$, since there
are 2 zeros before the first 1, giving $x_0^2$, 1 zero immediately before the second 1, giving
$x_1^1$, and no zeros immediately before the third 1 or the fourth one,
giving $x_2^0$ and $x_3^0$, and so altogether $x_0^2x_1^1x_2^0x_3^0$).
But the number of arrangements of $n$ ones and $t$ zeros is $\binom{t+n}{n}$.

\vskip\baselineskip\noindent {\bf Solution \ref{dimbound1pt}.} 
Let $q=(0,\ldots,0)\in\Aff{n}$. There is an automorphism $\psi:\Aff{n}\to\Aff{n}$
taking $p$ to $q$, given by translation $(b_1,\ldots,b_n)\mapsto(b_1-a_1,\ldots,b_n-a_n)$.
The corresponding automorphism on rings is $\psi^*:\field[X_1,\ldots,X_n]\to \field[X_1,\ldots,X_n]$
where $X_i\mapsto X_i+a_i$. Note that $\psi^*(I(q)^m)=I(p)^m$ and that 
$\psi^*$ induces vector space bijections 
$A_{\leq t}\to A_{\leq t}$ and $(I(q)^m)_{\leq t}\to(I(p)^m)_{\leq t}$. 
Thus it is enough to consider the case that $a_i=0$ for all $i$.
In this case $I=(X_1,\ldots,X_n)$ is a monomial ideal, and hence homogeneous.
Thus $\alpha(I^m)=m\alpha(I)=m$. Therefore, $t<m$ implies $H^\leq_{I^m}(t)=0$.
If $t<m$, let $t+i=m$ for some $i>0$. Then 
$\binom{t+n}{n}\leq\binom{m-1+n}{n}$ (look at Pascal's triangle)
so $\binom{t+n}{n}-\binom{m+n-1}{n}\leq 0$, hence
$H^\leq_{I^m}(t)\geq\binom{t+n}{n}-\binom{m+n-1}{n}$
with equality for $t=m-1$.
For $t\geq m$, $(I^m)_{\leq t}$ is spanned by the monomials of degree
$m$ through degree $t$. By introducing a variable $X_0$, we can regard these
as being monomials of degree exactly $t$ in $\field[X_0,\ldots,X_n]$ such that $X_0$ has 
exponent at most $t-m$: given any monomial $\mu$
in $X_1,\ldots,X_n$ of degree $m\leq i\leq t$, $X_0^{t-i}\mu$ is a monomial
in $X_0,\ldots,X_n$ of degree $t$ such that $X_0$ has exponent at most $t-m$. 
By Exercise \ref{starsandbars}, there are $\binom{t+n}{n}$ monomials in 
$X_0,\ldots,X_n$ of degree $t$. The monomials in $X_0,\ldots,X_n$ of degree $t$
but for which $X_0$ has exponent more than $t-m$ are in bijective correspondence
with the monomials in $X_0,\ldots,X_n$ of degree $m-1$ (just multiply by
$X_0^{t-m+1}$). There are thus $\binom{t+n}{n}-\binom{m-1+n}{n}$
monomials of degree $t$ in $X_0,\ldots,X_n$ for which $X_0$ has exponent at most
$t-m$, hence $H^\leq_{I^m}(t)=\binom{t+n}{n}-\binom{m+n-1}{n}$.

\vskip\baselineskip\noindent {\bf Solution \ref{exaffalpha}.}  
There is a linear polynomial $f$ defining the line through $p_1$ and $p_2$.
Thus $f^m\in I^m$ so $\alpha(I^m)\leq m=m\alpha(I)$ (see Exercise \ref{exaffalphapower}).
By Exercise \ref{dimbound1pt}, $H^\leq_{I(p_1)^m}(m-1)=0$ and $H^\leq_{I(p_1)^m}(m)>0$, so
$\alpha(I(p_1)^m)=m$. But $I\subset I(p_1)$ so $I^m\subset I(p_1)^m$ 
hence $m=\alpha(I(p_1)^m)\leq \alpha(I^m)$
so $\alpha(I^m)=m=m\alpha(I)$. Now consider the second statement.
Since $p_1,p_2,p_3$ are noncolinear, no linear polynomial can vanish at all three points.
Thus $\alpha(J)\geq 2$.
Let $f_1$ be the linear polynomial defining the line through $p_2$ and $p_3$,
$f_2$ the linear polynomial defining the line through $p_1$ and $p_3$,
and $f_3$ the linear polynomial defining the line through $p_1$ and $p_2$.
If $m=2s$, then $(f_1f_2f_3)^s$ has degree $3s=3m/2$ but is in 
$I(p_1)^m\cap I(p_2)^m\cap I(p_3)^m=J^m$ so
$\alpha(J^m)\leq 3m/2<2m\leq m\alpha(J)$. 
If $m=2s+1$, then $(f_1f_2f_3)^sf_1f_2\in J^m$, hence
$\alpha(J^m)\leq 3s+2<4s+2=2m\leq m\alpha(J)$.

\vskip\baselineskip\noindent {\bf Solution \ref{nondecrAff}.}  
We have a vector space inclusion $\phi:A_{\leq t}\to A_{\leq t+1}$. 
Compose with the quotient
$A_{\leq t+1}/I_{\leq t+1}$; the kernel is $I_{\leq t}$, hence $\phi$ induces an injective
map $A_{\leq t}/I_{\leq t}\to A_{\leq t+1}/I_{\leq t+1}$.

\vskip\baselineskip\noindent {\bf Solution \ref{dimbound}.}  
The polynomials in $I(p_i)^{m_i}$ of degree at most $t$ form a linear subspace of $A_{\leq t}$
defined by $\binom{m_i+n-1}{n}$ homogeneous linear equations.
Thus $(I(m_1p_1+\cdots+m_rp_r))_{\leq t}$ is a linear subspace defined by
$\sum_i\binom{m_i+n-1}{n}$ homogeneous linear equations.
Therefore $H^\leq_{I}(t)\geq\binom{t+n}{n}-\sum_i\binom{m_i+n-1}{n}$,
with the inequality (as opposed to equality) arising since the equations
need not be independent.

For the rest, note that by the Chinese Remainder Theorem we have an isomorphism
$A/I\to\bigoplus_i A/I(p_i)^{m_i}$ in which $f+I\mapsto (f+I(p_1)^{m_1},\ldots,f+I(p_r)^{m_r})$.
But $A/I(p_i)^{m_i}$ is finite dimensional for each $i$ (of dimension
$\binom{m_i+n-1}{n}$ in fact), so for some $d$ we have a surjection
$A_{\leq d}\to\oplus_i A/I(p_i)^{m_i}$ and hence a surjection 
$A_{\leq t}\to\oplus_i A/I(p_i)^{m_i}$ for all $t\geq d$.
Thus $A_{\leq t}/I_{\leq t}\cong \bigoplus_i A/I(p_i)^{m_i}$ for all $t\geq d$, so
$H^\leq_{A/I}(t)=\dim_\field(A_{\leq t}/I_{\leq t})=\sum_i\binom{m_i+n-1}{n}$ and
$H^\leq_I(t)=\binom{t+n}{n}-\sum_i\binom{m_i+n-1}{n}$.

\vskip\baselineskip\noindent {\bf Solution \ref{NagataBound}.} 
By Exercise \ref{dimbound}, $\alpha(I^m)\leq t$ if $\binom{t+n}{n}-r\binom{m+n-1}{n}>0$.
If we regard $\binom{t+n}{n}$ as being $(t+n)(t+n-1)\cdots(t+1)/n!$ and
$\binom{m+n-1}{n}$ as being $(m+n-1)\cdots(m+1)m/n!$,
then substitute $t=\lambda m$ (so $\lambda=t/m$); 
$\binom{t+n}{n}-r\binom{m+n-1}{n}$ becomes a polynomial in $m$ of degree
$n$ with leading coefficient $\lambda^n-r$. Thus,
for any integers $t$ and $m$ such that $t>m\sqrt[n]{r}$, 
$\binom{t'+n}{n}-r\binom{m'+n-1}{n}$ will be positive
for $t'=ti$ and $m'=mi$ for $i\gg0$. 
I.e., $\gamma(I)\leq \alpha(I^m)/m\leq t/m$, but we can choose integers $t$ and $m$
such that $t/m$ is arbitrarily close to but bigger than $\sqrt[n]{r}$,
hence $\gamma(I)\leq \sqrt[n]{r}$.

If $1\leq r\leq n$, the points lie on a hyperplane, so $\gamma(I)\leq \alpha(I)/1=1$.
But $\gamma(J)=1$ if $J$ is the ideal of any one of the points, so 
(as we saw for $r=n=2$ in the proof of Proposition \ref{valsofgammaProp} above)
$1=\gamma(J)\leq\gamma(I)$ so $\gamma(I)=1$.

\vskip\baselineskip\noindent {\bf Solution \ref{NagConj}.}  
Let $a=\inf\{\frac{t}{m}: \binom{t+n}{n}-s\binom{m+n-1}{n}>0; m,t\geq1\}$.
We can rewrite $\binom{t+n}{n}-s\binom{m+n-1}{n}>0$ 
as $\frac{(t+1)(t+2)}{2}-s\frac{m(m+1)}{2}>0$, which is equivalent to
$t^2+3t-s(m^2+m)\geq0$. This in turn becomes 
$m^2(l^2-s)+m(3l-s)\geq0$ if we substitute $t=lm$.
If $l=t/m<\sqrt{s}$ and $s\geq 9$, then $l^2-s<0$ and $3l-s<0$,
hence $m^2(l^2-s)+m(3l-s)<0$, so $t^2+3t-s(m^2+m)<0$
and therefore $\binom{t+n}{n}-s\binom{m+n-1}{n}\leq0$.
It follows that $a\geq\sqrt{s}$.
But if $l=t/m>\sqrt{s}$, then the leading coefficient $l^2-s$ of
$m^2(l^2-s)+m(3l-s)$ is positive, hence for $i\gg0$, 
$(im)^2(l^2-s)+im(3l-s)>0$. Therefore
$\binom{it+n}{n}-s\binom{im+n-1}{n}>0$, so 
$a\leq (ti)/(mi)=t/m=l$ for all rationals $l>\sqrt{s}$,
hence $a\leq\sqrt{s}$ so $a=\sqrt{s}$.

\vskip\baselineskip\noindent {\bf Solution \ref{units}.}  
Say $p=(a_1,\ldots,a_n)$. Then $A=\field[X_1,\ldots,X_n]=\field[Y_1,\ldots,Y_n]$, where
$Y_i=X_i-a_i$, and $I(p)=(X_1-a_1,\ldots,X_n-a_n)=(Y_1,\ldots,Y_n)$.
Given any element $f\in A$, it has the same degree whether expressed in terms of the variables
$Y_i$ or in terms of the $X_i$, but $A/(I(p))^m=A/(Y_1,\ldots,Y_n)^m$, so every element
of $A/(I(p))^m$ is the image of an element of degree at most $m-1$. Moreover, if $f(p)=0$,
then $f\in I(p)$, so $\overline{f}$ is nilpotent (since $\overline{f}^m=0$)
hence not a unit. And if $f(p)\neq0$, let $g=(f-f(p))/f(p)$.
Then $\overline{g}^m=0$, so $(1+\overline{g})(1+(-\overline{g})+(-\overline{g})^2+\cdots+(-\overline{g})^{m-1})=1$.
Thus $\overline{f}=f(p)(1+\overline{g})$ is a unit since $f(p)$ and $1+\overline{g}$ are units.

\vskip\baselineskip\noindent {\bf Solution \ref{AffIsIrr}.}  
If $f(p)=0$ for all $p\in\Aff{n}$, then $f\in\sqrt{(0)}$ by the Nullstellensatz,
hence $f=0$, contrary to assumption.

\vskip\baselineskip\noindent {\bf Solution \ref{pointsavoidance}.}  
If $n=1$, this is clear, so assume $n>1$.
For each $i$ and $j$, consider the vector $v_{ij}$ from $p_i$ to $p_j$.
Then it suffices to find $f$ such that $f(v_{ij})\neq0$ for all $i\neq j$;
i.e., given finitely many points $[v_{ij}]\in\pr{n-1}$, we must find a linear
form $f\in\field[\pr{n-1}]$ such that $f(v_{ij})\neq0$ for all $i\neq j$.
I.e., regarding linear forms as points in the dual space $(\pr{n-1})^*$
and points $v_{ij}$ as hyperplanes in $(\pr{n-1})^*$, we must find a point
in $(\pr{n-1})^*$ not on any of a finite set of hyperplanes. 
But we can think of a point of $(\pr{n-1})^*$ as giving a point of $\Aff{n}$
(unique up to multiplication by nonzero scalars) and vice versa, and we can think of 
hyperplanes in $(\pr{n-1})^*$ as giving codimension 1 linear
subspaces of $\Aff{n}$ and vice versa, so the result
follows from Exercise \ref{AffIsIrr}.

\vskip\baselineskip\noindent {\bf Solution \ref{dimbound2}.}  
Let $s=t-(m_1+\cdots+m_r-1)$ and
let $g$ be a degree 1 polynomial that does not vanish at any of the points $p_i$
(start with any $g$ with $\deg(g)=1$ and replace $g$ by $g-c$, where
$c\in\field\setminus \{g(p_1),\ldots,g(p_r)\}$).
By Exercise \ref{pointsavoidance} we can pick a linear form $f$ such that 
$f(p_i)\neq f(p_j)$ whenever $p_i\neq p_j$. Define $f_i=g^s\Pi_{j\neq i}(f-f(p_j))^{m_j}$.
Note that $\deg(f_i)=m_1+\cdots+m_r-m_i+s=t-(m_i-1)$, and that $f_i\in I(p_j)^{m_j}$ for all $j\neq i$,
but by Exercise \ref{units} $f_i$ maps to a unit $\overline{f_i}$ in $A/I(p_i)^{m_i}$ 
under the quotient homomorphism $\phi_i: A\to A/I(p_i)^{m_i}$. 

Given any element $(\overline{a_1},\ldots, \overline{a_r})\in \bigoplus_j A/I(p_j)^{m_j}$, 
we can by Exercise \ref{units} pick elements
$b_j\in A_{\leq(m_j-1)}$ such that $\phi_j(b_j)=\overline{f_j}^{-1}\overline{a_j}$, for $j=1,\ldots,r$.
Consider the homomorphism
$\phi: A\to \bigoplus_j A/I(p_j)^{m_j}$ defined by $\phi(h)=(\phi_1(h),\ldots,\phi_r(h))$. 
Taking $h=\sum_jf_jb_j$, we see $\phi(h)=(\overline{a_1},\ldots, \overline{a_r})$,
and since $\deg(h)\leq \max_j\{\deg(f_jb_j)\} = \max_j\{t-(m_j-1)+m_j-1\}=t$, 
we see that $\phi(A_{\leq t})=\bigoplus_j A/I(p_j)^{m_j}$. This gives the result, since
$\bigoplus_j A/I(p_j)^{m_j}=\phi(A_{\leq t})\equiv A_{\leq t}/I_{\leq t}$, hence $H^\leq_I(t)=
\dim(I_{\leq t})=\dim(A_{\leq t})-\dim(\bigoplus_j A/I(p_j)^{m_j})=
\binom{t+n}{n}-\sum_j\binom{m_j+n-1}{n}$.

Note that $H^\leq_{I}(t)=\binom{t+n}{n}-\sum_i\binom{m_i+n-1}{n}$ is equivalent to
$\phi|_{A_{\leq t}}$ being surjective. Thus to show that
$H^\leq_{I}(t)>\binom{t+n}{n}-\sum_i\binom{m_i+n-1}{n}$
it is enough to show $\phi|_{A_{\leq t}}$ is not surjective.
Suppose the points are colinear. Let $L$ be the line containing the points.
Then we have a commutative diagram
\begin{equation*}
\begin{matrix}
A & \to & \bigoplus_j A/I(p_j)^{m_j}\\
\downarrow & & \downarrow\\
\hbox to0in{\hss$\overline{A}=\field[L]=$}\field[X] & \to & \bigoplus_j \overline{A}/\overline{I}(p_j)^{m_j}\\
\end{matrix}
\end{equation*}
where $\overline{A}=\field[L]= \field[X]$ is the coordinate ring of the line $L$, hence a 
polynomial ring in a single variable $X$, and $\overline{I}(p_j)$ is the ideal in $\overline{A}$
of the point $p_j$. The upper horizontal arrow is $\phi$, the lower one is
the corresponding homomorphism $\overline{\phi}$ for dimension 1.
The vertical arrows are the usual quotients, and are therefore surjective.
Thus to show $\phi|_{A_{\leq t}}$ is not surjective
for $t<m_1+\cdots+m_r-1$, it is enough to show 
$(\field[X])_{\leq t}\to \bigoplus_j \overline{A}/\overline{I}(p_j)^{m_j}$
is not surjective; i.e., it is enough to consider the case $n=1$.
But then $H^\leq_{I}(t)\geq 0$, while $\binom{t+1}{1}-\sum_i\binom{m_i}{1}=t+1-\sum_im_i<0$ if $t<m_1+\cdots+m_r-1$.

\vskip\baselineskip\noindent{\bf 4. Fat points in projective space}

\vskip\baselineskip\noindent {\bf Solution \ref{homogprod}.}  
First, $I(p_1)^{m_1}\cdots I(p_r)^{m_r}\subseteq \cap_{i=1}^rI(p_i)^{m_i}$
and $\alpha(I(p_1)^{m_1}\cdots I(p_r)^{m_r})=
m_1+\cdots+m_r$. Now, by pairing the points up $p_1$ with $p_2$, $p_3$ with $p_4$,
etc.\ (there will be a point left over if $r$ is odd), 
we can pick a linear form that vanishes on $p_1$ and $p_2$, and a linear form
that vanishes on $p_3$ and $p_4$, etc.\ (if $r$ is odd, just pick any line through the 
leftover point). Raising the first to the power
$\max(m_1,m_2)$, the second to the power $\max(m_3,m_4)$, etc., and then
multiplying the results together we obtain a form of degree
$\max(m_1,m_2)+\max(m_3,m_4)+\cdots$ in $\cap_{i=1}^rI(p_i)^{m_i}$.
But  $\max(m_1,m_2)+\max(m_3,m_4)+\cdots <m_1+\cdots+m_r$,
hence $I(p_1)^{m_1}\cdots I(p_r)^{m_r}\neq\cap_{i=1}^rI(p_i)^{m_i}$.

\vskip\baselineskip\noindent {\bf Solution \ref{nondecrHF}.}  
Choose a linear form $F$ that does not vanish at any of the points $p_i$.
(This is always possible if the field $\field$ is large enough, but might not
be possible if $\field$ is finite.)
Let $G\in R_t$. By a linear change of coordinates, we may assume $F=x_0$.
Recall the map $\delta_t$ defined right after Remark \ref{notopenrem}.
If $FG\in I_{t+1}$, then $\delta_t(G)=\delta_{t+1}(x_0G)=\delta_{t+1}(FG)\in (I_A)_{\leq t+1}$,
but $\deg(\delta_t(G))\leq t$, so $\delta_t(G)\in (I_A)_{\leq t}$, hence
$G=\eta_t(\delta_t(G))\in (I_R)_t=I_t$. Thus multiplication by $F$ gives an injection
$(R/I)_t\to (R/I)_{t+1}$, hence $H_{R/I}(t)\leq H_{R/I}(t+1)$ for all $t\geq 0$.
(Alternatively, one could also approach this via a primary decomposition $I=\cap_i Q_i$.
The primes corresponding to the primary components of the
primary decomposition of $I$ are just the ideals $I(p_i)$ of the points;
i.e., $\sqrt{Q_i}=I(p_i)$. By hypothesis, $F\not\in I(p_i)$ for the points $p_i$,
hence for each $i$ we have $F^j\not\in Q_i$ for all $j\geq1$.
But $FG\in I$ implies $FG\in Q_i$ for all $i$; since no power of $F$ is in $Q_i$ we must
have $G\in Q_i$ for all $i$ hence $G\in I$. Thus multiplication by $F$ gives an injection
$R/I\to R/I$, and since $F$ is homogeneous of degree 1, this means
multiplication by $F$ gives an injection $(R/I)_t\to (R/I)_{t+1}$ for each $t\geq0$.)

\vskip\baselineskip\noindent {\bf Solution \ref{incrHF}.}  
By Exercise \ref{nondecrHF} or Exercise \ref{nondecrAff} we know that $H_{R/I}$ is nondecreasing. By Equation \eqref{affprofeq} and Exercise \ref{dimbound2},
$H_{R/I}(t)=\sum_i\binom{m_i+n-1}{n}$ for $t\gg0$.
Thus it is enough to show $H_{R/I}(s)=H_{R/I}(s+1)$ implies 
$H_{R/I}(s+1)=H_{R/I}(s+2)$ (and hence by induction $H_{R/I}(t)$ is constant,
and in fact equal to $\sum_i\binom{m_i+n-1}{n}$, for all $s\geq c$).

Choose linearly independent linear forms $F_0,\ldots, F_n$ such
that none of the $F_j$ vanish at any of the $p_i$.
By Exercise \ref{nondecrHF}, multiplication by any $F_j$ 
gives injective vector space homomorphisms $\lambda_{j,t}:R_t/I_t\to R_{t+1}/I_{t+1}$ for all $t\geq0$.
If $H_{R/I}(s)=H_{R/I}(s+1)$, then $\lambda_{j,s}$ is an isomorphism
for all $j$. Thus, since multiplication is commutative, for all $i$ and $j$ we have
$\lambda_{i,s+1}(R_{s+1}/I_{s+1})=
\lambda_{i,s+1}\lambda_{j,s}(R_s/I_s)=
\lambda_{j,s+1}\lambda_{i,s}(R_s/I_s)=
\lambda_{j,s+1}(R_{s+1}/I_{s+1})$.
But $F_0,\ldots,F_n$ generate $R$; in particular, $F_0R_t+\cdots+F_nR_t=R_{t+1}$ 
for all $t\geq0$, so $\sum_i F_i(R_t/I_t)=R_{t+1}/I_{t+1}$,
hence $R_{s+2}/I_{s+2}=\sum_j\lambda_{j,s+1}(R_{s+1}/I_{s+1})=
\sum_j\lambda_{i,s+1}(R_{s+1}/I_{s+1})=\lambda_{i,s+1}(R_{s+1}/I_{s+1})$,
so $H_{R/I}(s+2)=\dim \lambda_{i,s+1}(R_{s+1}/I_{s+1})=\dim R_{s+1}/I_{s+1}=H_{R/I}(s+1)$.
(Alternatively, let $F$ be a linear form not vanishing at any of the points. 
By Exercise \ref{nondecrHF}, $F$ induces an injection 
$(R/I)_t \to (R/I)_{t+1}$. So we have an exact sequence
$$0 \to (R/I)_t \stackrel{\times F}{\to} (R/I)_{t+1} \to (R/(I, F ))_{t+1} \to 0.$$
The module on the right is a standard graded algebra, so it cannot be zero in one degree
and nonzero in the next.)

\vskip\baselineskip\noindent {\bf Solution \ref{decrHF}.}  
Let $J=(x^2y,xy^3)$. Then $H_{R/J}=(1,2,3,3,2,2,\ldots)$.

\vskip\baselineskip\noindent {\bf Solution \ref{SHGHimpliesNag}.}  
Let $I=I_R(p_1+\cdots+p_r)\subset\field[\pr2]$ be the ideal
of $r\geq9$ generic points $p_i\in\pr2$.
By Exercise \ref{NagataBound}, we have $\gamma(I)\leq\sqrt{r}$.
By Conjecture \ref{SHGHconjSpecCase}, it is enough now to show for $r\geq 9$
and $m>0$ that 
$t<m\sqrt{r}$ implies $\binom{t+2}{2}-r\binom{m+1}{2}<1$, since then
$\alpha(I^{(m)})/m>t/m$ for all $t/m<\sqrt{r}$ and hence
$\gamma(I)=\lim_{m\to\infty}\alpha(I^{(m)})/m\geq\sqrt{r}$.
But $\binom{t+2}{2}-r\binom{m+1}{2}= (t^2+3t+2-r(m^2+m))/2$,
and $(t^2+3t+2-r(m^2+m))/2<0$ for $t=0$, so (since $(t^2+3t+2-r(m^2+m))/2$ is
strictly increasing as a function of $t$ for $t\geq 0$) it suffices now to show
$(t^2+3t+2-r(m^2+m))/2\leq1$ for $t=m\sqrt{r}$.
But $(t^2+3t+2-r(m^2+m))/2=1-m(r-3\sqrt{r})/2$ for $t=m\sqrt{r}$,
and $r-3\sqrt{r}\geq0$ for $r\geq 9$, so the result follows.

\vskip\baselineskip\noindent {\bf Solution \ref{W-Sk general bound}.}  
Since $I^{((m-1+n)t)}\subseteq (I^{(m)})^t$ we 
have 
$$t\alpha(I^{(m)})=\alpha((I^{(m)})^t)\leq\alpha(I^{(t(n+m-1))}),$$ 
so dividing by $t(n+m-1)$ 
and taking the limit as $t\to\infty$ gives
$$\frac{\alpha(I^{(m)})}{n+m-1}\leq\gamma(I).$$

\vskip\baselineskip\noindent {\bf Solution \ref{easyContainment}.}  
Say $r\geq m$; then $I(p_i)^r\subseteq I(p_i)^m$, hence 
$$I^r\subseteq I^{(r)}=\cap_i I(p_i)^r\subseteq \cap_i I(p_i)^m=I^{(m)}.$$
Conversely, assume $I^r\subseteq I^{(m)}$. 
By Equation \eqref{affprofeq}, for all $t$ we have
$H_{I_R^{(m)}}(t)=H^\leq_{I_A^m}(t)$ and $H_{I_R^{(r)}}(t)=H^\leq_{I_A^r}(t)$.
By Remark \ref{sympowerRem} we have 
$H_{I_R^r}(t)=H_{I_R^{(r)}}(t)$ for $t\gg0$, and thus $H_{I_R^r}(t)=H^\leq_{I_A^r}(t)$. By 
Exercise \ref{dimbound} for $t\gg0$ we have
$H^\leq_{I_A^r}(t)=\binom{t+n}{n}-s\binom{r+n-1}{n}$
and $H^\leq_{I_A^m}(t)=\binom{t+n}{n}-s\binom{m+n-1}{n}$.
But $I^r\subseteq I^{(m)}$ implies $(I^r)_t\subseteq (I^{(m)})_t$
and thus, for $t\gg0$ we have
\begin{align*}
\binom{t+n}{n}-s\binom{m+n-1}{n}&=H^\leq_{I_A^m}(t)=H_{I_R^{(m)}}(t)\\
&\geq H_{I_R^r}(t)=H^\leq_{I_A^r}(t)=\binom{t+n}{n}-s\binom{r+n-1}{n},
\end{align*}
hence $s\binom{m+n-1}{n}\leq s\binom{r+n-1}{n}$ for $t\gg0$.
But as is easy to see by looking at Pascal's triangle,
$\binom{j+n-1}{n}$ is an increasing function of $j$, so we conclude $m\leq r$.

\vskip\baselineskip\noindent{\bf 5. Examples: bounds on the Hilbert function of fat point subschemes of $\pr2$}

\vskip\baselineskip\noindent {\bf Solution \ref{diffOseqEx}.}  
Note that $\Delta H_{R/I(Z)}=\operatorname{diag}({\bf d})$
where ${\bf d}=(r_1,\ldots,r_s)$; the given answer
is just $\operatorname{diag}({\bf d})$ for this ${\bf d}$.
It is tedious to write this out; examine some dot diagrams.

\vskip\baselineskip\noindent {\bf Solution \ref{starConfigEx}.}  
Consider the reduction vector ${\bf d}=(9,8,7,6,3,2,1)$
obtained from the sequence of lines $L_0,L_1,L_2,L_3,L_0,L_1,L_2$.
Then 
$$\Delta H_{R/I(Z)}=\operatorname{diag}({\bf d})=(1,2,3,4,5,6,7,4,4,0,0,\ldots).$$

\vskip\baselineskip\noindent {\bf Solution \ref{BndOnReg}.}  
We obtain ${\bf d}$ by construction. Every dot in the dot diagram which we use to compute
$\operatorname{diag}({\bf d})$ is on a diagonal line with 
$x$-intercept at most $m_1+\cdots+m_r-1$,
hence for $t\geq m_1+\cdots+m_r-1$, $H_{R/I(Z)}(t)\geq N$, 
where $N$ is the total number of dots,
but the number of dots is $N=\sum_i\binom{m_i+1}{2}$, and we know
$\sum_i\binom{m_i+1}{2}\geq H_{R/I(Z)}(t)$ for all $t\geq 0$. 
Thus $\sum_i\binom{m_i+1}{2}\geq H_{R/I(Z)}(t)$ for all $t\geq 0$ and
$H_{R/I(Z)}(t)\geq \sum_i\binom{m_i+1}{2}$ for $t\geq m_1+\cdots+m_r-1$,
so $H_{R/I(Z)}(t)=\sum_i\binom{m_i+1}{2}$ for all
$t\geq m_1+\cdots+m_r-1$.

\vskip\baselineskip\noindent{\bf 6. Hilbert functions: some structural results}

\vskip\baselineskip\noindent {\bf Solution \ref{exO-seq1}.}  
One solution is to use Theorem \ref{CHTtheorem} to find a reduction vector ${\bf d}$.
Pick a line $L$ through $p$ and let $L_1=L_2=L_3=L$.
The corresponding reduction vector is ${\bf d}=(3,2,1)$.
By the theorem, $\Delta H_{R/I}=(1,2,3,0,0,\ldots)$. But if we pick three distinct lines
$L_1',L_2',L_3'$ and on $L_1'$ we pick 3 points, on
$L_2'$ we pick 2 points and on $L_3'$ we pick 1 point,
where we avoid ever picking a point where two lines cross,
and if we let $Z$ be the union of these six points, then
by the theorem $\Delta H_{R/I(Z)}=\operatorname{diag}({\bf d})$.
Hence, $I$ and $I(Z)$ have the same Hilbert function.
Alternatively, it is not hard to work out the Hilbert function of a 
power of the ideal of a single point. Doing so gives
$H_{R/I}=(1,3,6,6,6,\ldots)$, so
another solution is to work backwards to find ${\bf d}$, given the fact 
asserted by Theorem \ref{GMRthm} that
$\Delta H_{R/I}=\operatorname{diag}({\bf d})$. 
Thus $\Delta H_{R/I}=(1,2,3,0,0,\ldots)$ so ${\bf d}=(3,2,1)$.
Now proceed as in the first solution to obtain $Z$ with
Hilbert function $H_{R/I}$.

\vskip\baselineskip\noindent {\bf Solution \ref{exO-seq2}.}  
It is enough to show that ${\bf d}\neq{\bf d}'$ implies
that $\operatorname{diag}({\bf d})\neq \operatorname{diag}({\bf d}')$.
Say ${\bf d}=(d_1,\ldots,d_r)$ and ${\bf d}'=(d_1',\ldots,d_s')$.
By assumption, ${\bf d}$ and ${\bf d}'$ are decreasing.
We will prove the contrapositive, so assume
$\operatorname{diag}({\bf d})=\operatorname{diag}({\bf d}')$.
If $d_1<d_1'$, then $d_i<d_1<d_1'$, for all $1<i\leq s$,
hence the entries of $\operatorname{diag}({\bf d})$ for degrees $t$ with
$d_1\leq t<d_1'$ will be 0 but nonzero for $\operatorname{diag}({\bf d}')$
(where we note the degree 0 entry is the first entry, the degree 1 entry is the second
entry, etc.). Thus $d_1\leq d_1'$ and by symmetry we have $d_1=d_1'$.
Therefore after deleting the first entries of ${\bf d}$ and ${\bf d}'$ we get
$\operatorname{diag}((d_2,\ldots,d_r))=\operatorname{diag}((d_2',\ldots,d_s'))$
and we repeat the argument. Eventually we obtain $d_i=d_i'$ for all $i$,
and so $r=s$, and thus ${\bf d}={\bf d}'$.

\vskip\baselineskip\noindent{\bf 7. B\'ezout's theorem in $\pr2$ and applications}

\vskip\baselineskip\noindent {\bf Solution \ref{intmultEx1}.} 
Say $F(p)\neq 0$. Then $F\not\in I(p)^m$ for $m\geq1$, so 
$$(x_0,x_1,x_2)\subseteq \sqrt{(F,G)+I(p)^m}$$
by the Nullstellensatz, hence for $t$ large enough $(x_0,x_1,x_2)_t=((F,G)+I(p)^m)_t$
so $\dim R_t/((F,G)+I(p)^m)_t=0$.

\vskip\baselineskip\noindent {\bf Solution \ref{intmultEx2}.} 
Since $x_0$ does not vanish at $p$, we have $I_p(x_0F,G)=I_p(F,G)+I_p(x_0,G)=I_p(F,G)$.
But $(x_0F,G)=(x_0F-G,G)$, so $I_p(x_0F-G,G)=I_p(x_0F,G)$.
But $x_0F-G=x_2^2(x_2-x_0)$ so $I_p(x_0F-G,G)=I_p(x_2^2(x_2-x_0),G)
=I_p(x_2^2,G)+I_p(x_2-x_0,G)=2I_p(x_2,G)+0=2\cdot1$.

To compute $\sum_{p\in\pr2}I_p(F,G)$, it's enough to consider only those points $p\in\pr2$
where both $F$ and $G$ vanish; i.e.,
$\sum_{p\in\pr2}I_p(F,G)=I_{(1,0,0)}(F,G)+I_{(0,1,0)}(F,G)+I_{(1,1,1)}(F,G)$.
We just found $I_{(1,0,0)}(F,G)=2$. Similarly, we find $I_{(0,1,0)}(F,G)=3$.
At $p=(1,1,1)$, the tangent to $F$ at $p$ is $x_1-2x_2$ and the tangent to $G$ at $p$
is $x_1-3x_2$. These are different, so $I_{(1,1,1)}(F,G)=\mult_p(F)\mult_p(G)=1$,
hence $\sum_{p\in\pr2}I_p(F,G)=6=\deg(F)\deg(G)$.

\vskip\baselineskip\noindent {\bf Solution \ref{intmultEx3}.} 
See \cite[Example 4.2.3]{refCHT} or \cite[Lemma 8.4.7]{refB. et al}.

\vskip\baselineskip\noindent {\bf Solution \ref{gammafor4pts}.} 
Consider $I^{(m)}=I(m(p_1+p_2+p_3+p_4))$.
Suppose $0\neq F\in (I^{(m)})_{2m-1}$.
Note that $F$ vanishes to order at least $m$ at each of two points
on any line $L_{ij}$ through two of the points $p_i, p_j$.
Since $2m>2m-1$, this means by B\'ezout  that the linear forms (also denoted $L_{ij}$)
defining the lines are factors of
$F$. Dividing $F$ by $B=L_{12}L_{13}L_{14}L_{23}L_{24}L_{34}$ we obtain a form $G$ of degree $2m-7$
in $(I^{(m-3)})_{2m-7}$. The same argument applies: $B$ must divide $G$.
Eventually we obtain a form of degree less than 6 divisible by $B$, which is impossible.
Thus $F=0$, and $\alpha(I^{(m)})> 2m-1$. Since $(L_{12}L_{34})^m\in I^{(m)}$,
we see that $\alpha(I^{(m)})\leq 2m$, thus $\alpha(I^{(m)})=2m$.

\vskip\baselineskip\noindent {\bf Solution \ref{gammafor5pts}.} 
Since $H_I(2)\geq \binom{2+2}{2}-5=1$, there is a nonzero form $F\in I_2$,
hence $\gamma(I)\leq \alpha(I)/1=2$.
If $F$ were reducible, it would be a product of two linear forms, and hence
three of the points would be colinear. Thus $F$ is irreducible.
Now let $0\neq G\in I^{(m)}_{2m-1}$. By B\'ezout, $F$ and $G$ have a common factor,
but $F$ is irreducible, so $F|G$; say $FB=G$, hence $B\in  I^{(m-1)}_{2(m-1)-1}$.
Again we see that $F|B$, etc. Eventually we find that $F$ divides a 
form of degree less than 2, which is impossible. 
Thus $I^{(m)}_{2m-1}=0$, so $\alpha(I^{(m)})\geq 2m$, so $\gamma(I)\geq 2m/m=2$.

\vskip\baselineskip\noindent {\bf Solution \ref{6ptsNotOnaConic}.} 
Pick 5 points $p_1,\ldots,p_5$ on an irreducible conic $C$, defined by an irreducible form $F$.
Note that no three of these five points are colinear (else the line through the three is a component
of the conic, which can't happen since the conic is irreducible).
Pick any point $p_6$ not on $C$ and not on any line through any two of the other points. 
If there were a nonzero form $G$ of degree 2 such that
$G$ vanished at all six points, then $F|G$ by B\'ezout, hence $G$ is a constant times $F$, so
$F$ would also have to vanish at $p_6$.

\vskip\baselineskip\noindent {\bf Solution \ref{gammafor6pts}.} 
By Proposition \ref{valsofgammaProp} we know $\gamma(I)\leq 12/5$.
As in the solution to Exercise \ref{gammafor5pts}, there is an irreducible
form of degree 2 which vanishes at any five of the six points, and by hypothesis
each such form does not vanish at the sixth point. Let $F_i$ be the degree 2 form 
that vanishes at all of the points but $p_i$. Thus $F=F_1\cdots F_6\in (I^{(5)})_{12}$.
Say $0\neq G\in (I^{(5m)})_{12m-1}$. Then B\'ezout implies that
each $F_i$ divides $G$, hence $F|G$, so $B=G/F\in (I^{(5(m-1))})_{12(m-1)-1}$.
The argument can be repeated, and eventually we find that $F$ divides a form of degree less than
the degree of $F$. Hence $(I^{(5m)})_{12m-1}=0$, so $\alpha(I^{(5m)})\geq 12m$, so
$\gamma(I)\geq12/5$.

\vskip\baselineskip\noindent {\bf Solution \ref{7generalpts}.} 
Pick 6 points $p_1,\ldots,p_6$ which do not all lie on any conic, and no three of which are colinear.
Any conic through any five of the points is irreducible, otherwise there is a line through
3 or more of the points. There is also at most one conic through any given five of the 
points, by B\'ezout's Theorem. (Alternatively, if there are two conics through the same five points, some linear combination of the forms defining
the conics would give a form vanishing at all 6 points.)
Now pick any seventh point $p_7$ not on any conic through 5 of the points $p_1,\ldots,p_6$
and not on any line through any two of the points $p_1,\ldots,p_6$.
Clearly, no three of the points $p_1,\ldots,p_7$ can be colinear (since no three of
$p_1,\ldots,p_6$ are and since $p_7$ is not on any of the lines through 
two of the points $p_1,\ldots,p_6$), and by the same argument
no six of the points $p_1,\ldots,p_7$ can be contained in any conic.

\vskip\baselineskip\noindent {\bf Solution \ref{gammafor7pts}.} 
By Proposition \ref{valsofgammaProp} we know $\gamma(I)\leq 21/8$.
Since $\binom{3+2}{2}-\binom{2+1}{2}-6\binom{1+1}{2}>0$, there is for each $i$ a form $F_i$
of degree 3 that vanishes at each point $p_j$ but has multiplicity at least 2 at $p_i$.
If $F_i$ were reducible, it would either consist of a line and a conic
with $p_i$ at a point where the two meet, and then the remaining six points would
have to be put on the line or the conic, so either the line would have 3 or the conic would have 6,
contrary to hypothesis, or $F_i$ would consist of three lines, with one point where two of the lines meet,
and the other six points placed elsewhere on the three lines, but then one of the lines would have to
contain at least three of the points. Since $F_i$ is irreducible, it must have multiplicity exactly
2 at $p_i$ and 1 at the other points, 
otherwise by B\'ezout the line through $p_i$ and any $p_j$, $j\neq i$, would be a component.

Note $F=F_1\cdots F_7\in (I^{(8)})_{21}$. As usual, if there is a $G$ with $0\neq G\in (I^{(8m)})_{21m-1}$
we get a contradiction by repeated applications of B\'ezout. 
Thus $\alpha(I^{(8m)})\geq 21m$, hence $\gamma(I)\geq 21/8$.

\vskip\baselineskip\noindent {\bf Solution \ref{9generalpts}.} 
Clearly, $\alpha(I^{(m)})\leq 3m$, since $F^m\in(I^{(m)})_{3m}$, where $F$ is the cubic form defining $C$.
If $0\neq G\in (I^{(m)})_{3m-1}$, then $F|G$ by B\'ezout, and we get $B\in (I^{(m-1)})_{3(m-1)-1}$.
Repeating this argument we eventually get a form of degree less than 
that of $F$ which $F$ divides.
Hence $(I^{(m)})_{3m-1}=0$, so $\alpha(I^{(m)})\geq 3m$, so $\gamma(I)=3$.

\vskip\baselineskip\noindent{\bf 8. Divisors, global sections, the divisor class group and fat points}

\vskip\baselineskip\noindent {\bf Solution \ref{isometry}.} 
It is enough to check this for $w=s_i$ for all $i$, for the generators 
$s_i$ of $W_r$ given above.
But $s_i(x)\cdot s_i(y)=(x+(x\cdot n_i)n_i)\cdot (y+(y\cdot n_i)n_i)
= x\cdot y+2(y\cdot n_i)(x\cdot n_i)+(x\cdot n_i)(y\cdot n_i)(n_i\cdot n_i)
= x\cdot y$, and $n_i\cdot K_X=0$ for all $i$, so 
$s_i(K_X)=K_X+(K_X\cdot n_i)n_i=K_X$.

\vskip\baselineskip\noindent {\bf Solution \ref{RRexercise}.} 
We have 
$$\frac{F^2-K_X\cdot F}{2}+1=(t^2-\sum_im_i^2+3t-\sum_im_i)/2+1=(t^2+3t+2)/2-\sum_i(m_i^2+m_i)/2$$
which is just $\binom{t+2}{2}-\sum_i\binom{m_i+1}{2}.$
The rest is clear.

\vskip\baselineskip\noindent {\bf Solution \ref{homaloidalclass}.} 
The class $16e_0-6e_1-\cdots-6e_8$ reduces by $W_8$ to
$2e_0-6e_1-2e_2$, but a conic can't vanish to order more than 2 at a point.
Thus $\dim I(6p_1+\cdots+6p_8)_{16}=\dim I(3p_1+p_2)_1=0$, hence 
$\alpha(I(6p_1+\cdots+6p_8))\geq 17$.
However, $H_{I^{(6)}}(17)\geq \binom{19}{2}-8\binom{7}{2}=3>0$,
thus $\alpha(I(6p_1+\cdots+6p_8))\leq17$, hence $\alpha(I(6p_1+\cdots+6p_8))=17$.
Alternatively, $17e_0-6e_1-\cdots-6e_8$ reduces by $W_8$ to $e_0$,
hence $\dim I(6p_1+\cdots+6p_8)_{17}=\dim I(0)_1=\dim R_1=3$,
and we achieve the same conclusion.

\vskip\baselineskip\noindent {\bf Solution \ref{exceptionalclass}.} 
We have $C^2=e_1^2=-1=e_1\cdot K_X=C\cdot K_X$. Since $E_1$ is a smooth
rational curve, so is $C$.
Moreover, $((mC)^2-K_X\cdot (mC))/2+1=(-m^2+m+2)/2\leq0$ for all $m\geq 2$.

\vskip\baselineskip\noindent {\bf Solution \ref{fcexercise}.} 
It follows from B\'ezout's Theorem that $h^0(X,\OO_X(D))=h^0(X,\OO_X(F))$.
But $C\cdot F=0$, so $D^2=(F+mC)^2=F^2-m^2$ and 
$-K_X\cdot D=-K_X\cdot (F+mC)=-K_X\cdot F + m$,
so $(D^2-K_X\cdot D)/2=(F^2-K_X\cdot F)/2-(m^2-m)/2<(F^2-K_X\cdot F)/2$.
Thus $h^0(X,\OO_X(D))=h^0(X,\OO_X(F))\geq (F^2-K_X\cdot F)/2+1>(D^2-K_X\cdot D)/2+1$.

\vskip\baselineskip\noindent{\bf 9. The SHGH Conjecture}

\vskip\baselineskip\noindent {\bf Solution \ref{HilbFuncExercise}.} 
For $t<29$, the class of $D=tE_0-12E_1-12E_2-10E_3-\cdots-10E_8$ reduces via $W_8$ to
a divisor class where $e_0$ has a negative coefficient, so $H_I(t)=0$ for $t\leq 28$.
For $t=29$ we get $D'=E_0-2E_8$, so $H_I(29)=h^0(X,\OO_X(E_0))=3$.
For $t>29$, $D\cdot C\geq 0$ for all exceptional $C$, so 
$H_I(t)=\max(0,(D^2-K_X\cdot D)/2+1)$ and $(D^2-K_X\cdot D)/2+1$ turns out to
be positive so we have $H_I(t)=(D^2-K_X\cdot D)/2+1=\binom{t+2}{2}-2\binom{12+1}{2}-6\binom{10+1}{2}$
for all $t\geq30$.

\end{document}